\documentclass[12pt,twoside,reqno]{amsart}

\usepackage{geometry}
\usepackage{setspace}

\usepackage[dvipsnames]{xcolor} 

\usepackage{enumitem} 
\usepackage{comment}

\usepackage{mathtools,amssymb,amsthm,thmtools}
\usepackage{nicematrix}
\NiceMatrixOptions{cell-space-limits=5pt}
\makeatletter
\renewcommand*\env@matrix[1][\arraystretch]{%
  \edef\arraystretch{#1}%
  \hskip -\arraycolsep
  \let\@ifnextchar\new@ifnextchar
  \array{*\c@MaxMatrixCols c}}
\makeatother

\usepackage{thm-restate}
\newcommand\numberthis{\addtocounter{equation}{1}\tag{\theequation}}

\usepackage{hyperref,cleveref}
\hypersetup{
    colorlinks=true,
    citecolor=Green,
    linkcolor=blue,
    filecolor=magenta,      
    urlcolor=cyan,
    pdftitle={},
    pdfpagemode=FullScreen,
}

\usepackage[backend=biber,style=alphabetic,maxbibnames=5]{biblatex}
\newtheorem{thrm}{Theorem}[section]

\newtheorem{nota}[thrm]{Notation}
\newtheorem{lem}[thrm]{Lemma}
\newtheorem{cor}[thrm]{Corollary}
\newtheorem{conj}[thrm]{Conjecture}
\newtheorem{obs}[thrm]{Observation}

\newtheorem{exam}[thrm]{Example}

\newcommand{\F}{\mathbb{F}}

\usepackage{tikz} 
\usepackage{tikz-cd} 

\title[Colon Ideal Generators \& WLP for Monomial ACI in $3$ Variables]{The Generators of a Colon Ideal with an Application to the Weak Lefschetz Property for Monomial Almost Complete Intersections in Three Variables}

\author[Matthew D. Booth]{Matthew Davidson Booth}
\author[Adela Vraciu]{Adela Vraciu}

\address{Department of Mathematics, University of South Carolina, 
Columbia, SC 29208}
\email{mdbooth@email.sc.edu}
\email{vraciu@math.sc.edu}

\subjclass[2020]{Primary: 13C05, 13E10; Secondary: 05E40, 13D40}
\keywords{Artinian graded algebra, Hilbert function, level ideal, monomial almost complete intersection, weak Lefschetz property}

\begin{document}

\begin{abstract}
Much progress has been made in classifying when the weak Lefschetz property holds for $A=\F[x,y,z]/I$ where $\text{char}(\F)=0$ and $I=(x^{d_{1}},y^{d_{2}},z^{d_{3}},x^{a_{1}}y^{a_{2}}z^{a_{3}})$ is a monomial almost complete intersection. We connect this problem to the setting of two variables through a certain relation. In so doing, we are led to determine explicit formulas for the generators of the colon ideal $(x^{d_{1}},y^{d_{2}}):(x+y)^{a_{3}}$. With these generators in hand, we construct a matrix and show that failure of WLP for $A$ is dictated by the vanishing of a certain polynomial (namely the determinant of our matrix) when $A$ is level. We further show in the level case that a conjecture first posed in \cite{Mig-Mir-Nag-11} is true in a few new cases.
\end{abstract}

\maketitle
\tableofcontents

\section{Introduction}

Let $P=\F[x_{1},\dots,x_{n}]$ where $\F$ is a field of characteristic zero, and let $I$ be an Artinian homogeneous ideal of $P$. Recall that $P/I$ is said to enjoy the \textit{weak Lefschetz property (WLP)} if the multiplication by a general linear form map $\times\ell:(P/I)_{d-1}\rightarrow (P/I)_{d}$ has maximal rank at all degrees $d$. The first ``landmark'' result on the subject of WLP was given in \cite{Stanley-80} (see also \cite{Watanabe}), where it is shown that Artinian monomial complete intersection ideals always enjoy WLP. (In fact, they are shown to have the \textit{strong Lefschetz property, SLP}.) While arbitrary complete intersection ideals are not known to always have WLP, it is conjectured that they do. More recent work has given attention to the class of almost complete intersection ideals, which are on the whole more mysterious. Any almost complete intersection (and more generally any Artinian ideal) in two variables is known to have WLP from Proposition 4.4 in \cite{Har-Mig-Nag-Wat}, so the first case of interest is the three-variable setting.

Let $P=\F[x,y,z]$ and $I=(x^{d_{1}},y^{d_{2}},z^{d_{3}}, x^{a_{1}}y^{a_{2}}z^{a_{3}})$. The natural question is as follows: what conditions on the exponents defining $I$ characterize the presence of the WLP for $P/I$? Early results specific to this type of ideal are Corollary 7.3 in \cite{Brenner} and Theorem 3.3 in \cite{Brenner-Kaid-07}, where the authors determine semistability of the syzygy bundle associated with the ideal and show how this semistability implies the presence or absence of the WLP. Accompanying this analysis are a handful of numerical criteria which appear in subsequent treatments where ideals of the prescribed form are the focal point; two prominent such places are section six in \cite{Mig-Mir-Nag-11} and sections three and four in \cite{Cook-Nagel-21}. These last two sources form the primary point of departure for this paper.

Retaining the setup established thus far, if at least two of $a_{1}$, $a_{2}$, and $a_{3}$ are zero, then $I$ is a complete intersection (and so enjoys WLP). In the case where exactly one of $a_{1}$, $a_{2}$, or $a_{3}$ is zero, it was shown by Migliore, Mir\'o-Roig, and Nagel in \cite{Mig-Mir-Nag-11} (Lemma 6.6) that $P/I$ has WLP as well. This leaves only the possibility that all of $a_{1}$, $a_{2}$, and $a_{3}$ are nonzero for examination, and Theorem 4.10 of \cite{Cook-Nagel-21} takes a large step toward classifying when WLP holds or fails for such ideals.

The main results of the present paper concern level almost complete intersections. By part (iii) of Proposition 6.1 in \cite{Mig-Mir-Nag-11}, these are precisely the algebras satisfying $d_{1}-a_{1}=d_{2}-a_{2}=d_{3}-a_{3}$. The following was originally proposed as part (b) of Conjecture 6.8 in \cite{Mig-Mir-Nag-11}; letting $t$ denote the common value of the differences $d_{i}-a_{i}$, this conjecture appears again as Conjecture 4.13 in \cite{Cook-Nagel-21} to incorporate the two known ``rogue'' cases $(a_{1},a_{2},a_{3},t)\in\big\{(2,9,13,9),\, (3,7,14,9)\big\}$. Specifically, we have:

\begin{conj}\label{cook-nagel}
Consider the ideal $I=(x^{a_1+t}, y^{a_2+t}, z^{a_3+t}, x^{a_1}y^{a_2}z^{a_3})\subseteq\F[x, y, z]$ where $a_1, a_2, a_3, t$ satisfy the following three conditions:
\begin{equation}\label{conditions}
    0<a_1\le a_2 \le a_3\le 2(a_1+a_2),\ \ 3\ \text{divides}\ a_1+a_2+a_3, \ \ \text{and} \ \ t\ge \frac{a_1+a_2+a_3}{3}.
\end{equation}
If $(a_1, a_2, a_3, t)$ is not $(2, 9, 13, 9)$ or $(3, 7, 14, 9)$, then $\F[x, y, z]/I$ fails WLP if and only if $t$ is even, $a_1+a_2+a_3$ is odd, and $a_1$, $a_2$, $a_{3}$ are not all distinct.
\end{conj}

The WLP is known to hold when the assumptions in the conjecture are not satisfied. Indeed, part (a) of Conjecture 6.8 in \cite{Mig-Mir-Nag-11} says WLP holds when any of the three conditions in \eqref{conditions} does not hold, and this part of the conjecture is true by Lemma 6.6, Corollary 6.3, and Lemma 6.7 of the same. The failure of WLP for $(a_1, a_2, a_3, t)$ when $t$ is even, $a_1+a_2+a_3$ is odd, and $a_1=a_2$ or $a_2=a_3$ is given by Corollary 7.4 of \cite{Mig-Mir-Nag-11} (and is recovered by part (b)(3)(V') of Theorem 4.10 in \cite{Cook-Nagel-21}.)

The forward implication of the conjecture is still open, although some cases have been established (see for instance Proposition 4.14 and Corollary 4.16 of \cite{Cook-Nagel-21}). More precisely, the remaining open cases for level algebras are:
\begin{itemize}
    \item $a_1+a_2+a_3$ is odd, $t$ is even, and $a_1<a_2<a_3$;
    \item $a_1+a_2+a_3$ is even, $t$ is odd, and $a_1<a_2$ or $a_2<a_3$.
\end{itemize}

The goal of this paper is to make some progress toward resolving these open cases in Section 3. 
The prelude to these results, however, is a theorem of interest all on its own. Given $P=\F[x,y,z]$ and the ideal $I=(x^{d_{1}},y^{d_{2}},z^{d_{3}},x^{a_{1}}y^{a_{2}}z^{a_{3}})$ with $0<a_{j}<d_{j}$ for $j\in\{1,2,3\}$, consider the homomorphism which maps $z$ to $-(x+y)$. The image of $I$ in $\F[x,y]$ under this map is the ideal $\big(x^{d_{1}},y^{d_{2}},(x+y)^{d_{3}},x^{a_{1}}y^{a_{2}}(x+y)^{a_{3}}\big)$. A homogeneous relation with respect to this ideal then takes the form
\begin{equation*}
    Ax^{d_{1}}+By^{d_{2}}+C(x+y)^{d_{3}}+Dx^{a_{1}}y^{a_{2}}(x+y)^{a_{3}}=0\quad\text{where}\quad A,B,C,D\in \F[x,y].
\end{equation*}
The preceding line indicates that $C(x+y)^{d_{3}-a_{3}}+Dx^{a_{1}}y^{a_{2}}\in (x^{d_{1}},y^{d_{2}}):(x+y)^{a_{3}}$, and the generators of this colon ideal turn out to be useful in settling some of the open cases in \autoref{cook-nagel}. Section 2 is dedicated to proving \autoref{thrm:colon generators in general}, which provides explicit formulas for the generators of the aforementioned colon ideal. Knowing the degrees of these generators will allow us in Section 3 to prove \autoref{polynomials}, which shows that the vanishing of a certain determinant (which is a polynomial in $t$) coincides with failure of WLP for $I=(x^{a_{1}+t},y^{a_{2}+t},z^{a_{3}+t},x^{a_{1}}y^{a_{2}}z^{a_{3}})$ when $t\geq\frac{a_{1}+a_{2}+a_{3}}{3}$. Then Section 3 will conclude (see \autoref{borderline}) by showing that \autoref{cook-nagel} is true when $t=\frac{a_{1}+a_{2}+a_{3}}{3}+1$ provided $a_{3}=2(a_{1}+a_{2})-3a$ and $0\leq a\leq 3$.

The subsequent sections have many parameters to keep track of, so we mention here that all notation we define holds precisely in the section in which it is defined. For instance, the parameter ``$a$'' will play the role of $a_{3}$ in Section 2 but has a different meaning (which we specify) when it appears again in \autoref{borderline}.

In closing the introduction, we remark that the assumption $\text{char}(\F)=0$ is critical at many points in this paper; in particular, we shall routinely use the fact (which can easily be false in positive characteristic) that WLP holds for monomial complete intersections. Certain methods like translating our three-variable problem into a two-variable problem (e.g. \autoref{lem1}) could potentially be adapted to positive characteristic, but we do not generally expect our results to hold verbatim in that setting. Thus, unless otherwise specified, the reader should assume that all fields in our results have characteristic $0$.

\section{Generators of the Colon Ideal}

Fix the polynomial ring $P=\F[x,y]$ where $\F$ is a field of characteristic zero. Binomial coefficients $\binom{m}{r}$ for $r\geq 0$ are interpreted as polynomials in $m$ of degree $r$. The main result of this section is a description of the generators for the colon ideal $(x^{d_{1}},y^{d_{2}}):(x+y)^{a}$. Treating empty sums as $0$, we define four elements of $P$ which take center stage in what follows.

\begin{nota}\label{formulas}
Let $d\geq 2$ be an integer, and let $a$ be a positive integer. For each nonnegative integer $n$, define the following elements of $P$:
\begin{enumerate}
    \item When $a$ is odd,
    \begin{align*}
        F_{1,d,a,n} & =\sum_{i=0}^{d-\frac{a+1}{2}-n}(-1)^{i}\binom{d-1-i}{\frac{a-1}{2}+n}\binom{\frac{a-1}{2}+i}{\frac{a-1}{2}}x^{d-\frac{a+1}{2}-n-i}y^{i},\ \ \text{of degree}\ d-\frac{a+1}{2}-n\, ;\\
        F_{2,d,a,n} & =y^{2}\sum_{i=0}^{d-\frac{a+3}{2}-n}(-1)^{i}\binom{d-3-i}{\frac{a-3}{2}+n}\binom{\frac{a+1}{2}+i}{\frac{a+1}{2}}x^{d-\frac{a+3}{2}-n-i}y^{i},\ \ \text{of degree}\ d-\frac{a-1}{2}-n.
    \end{align*}
    Since $\frac{a-3}{2}+n<0$ when $a=1$ and $n=0$, specially define $F_{2,d,1,0}=(-1)^{d-2}(d-1)y^{d}$.
    \item When $a$ is even,
    \begin{align*}
        G_{1,d,a,n} & =y\sum_{i=0}^{d-\frac{a+2}{2}-n}(-1)^{i}\binom{d-2-i}{\frac{a-2}{2}+n}\binom{\frac{a}{2}+i}{\frac{a}{2}}x^{d-\frac{a+2}{2}-n-i}\,y^{i},\ \ \text{of degree}\ d-\frac{a}{2}-n\, ;\\
        G_{2,d,a,n} & =\sum_{i=0}^{d-\frac{a}{2}-n}(-1)^{i}\binom{d-1-i}{\frac{a-2}{2}+n}\binom{\frac{a-2}{2}+i}{\frac{a-2}{2}}(i-1)x^{d-\frac{a}{2}-n-i}\,y^{i},\ \ \text{of degree}\ d-\frac{a}{2}-n.
    \end{align*}
\end{enumerate}
\end{nota}

(Note that the asserted degrees of these forms assume the corresponding summation is nonempty.) Our goal in this section is to prove the following theorem.

\begin{restatable}{thrm}{generalgens}\label{thrm:colon generators in general}
Consider $(x^{d_{1}},y^{d_{2}}):(x+y)^{a}$ where $2\leq d_{1}\leq d_{2}$ and $1\leq a\leq d_{1}+d_{2}-2$. Setting $k=d_{2}-d_{1}$, the aforementioned colon ideal has the following generators in $P$:
\begin{enumerate}
    \item If $1\leq a\leq k$, then the generators are $x^{d_{1}}$ and
    \begin{equation*}
        H_{d_{1},a,k}:=y^{k-a+1}\sum_{i=0}^{d_{1}-1}(-1)^{i}\binom{d_{1}+a-2-i}{a-1}x^{d_{1}-1-i}y^{i},\ \text{of degree}\ \deg(H_{d_{1},a,k})=d_{2}-a.
    \end{equation*}
    \item If $k+1\leq a\leq d_{1}+d_{2}-2$, then:
    \begin{enumerate}[label=(\roman*)]
        \item When $a-k$ is odd, the generators are $F_{1,d_{2},a-k,k}$ and $F_{2,d_{2},a-k,k}$.
        \item When $a-k$ is even, the generators are $G_{1,d_{2},a-k,k}$ and $G_{2,d_{2},a-k,k}$.
    \end{enumerate}
\end{enumerate}
\end{restatable}

Note that $(x+y)^{d_1+d_2-1} \in (x^{d_1}, y^{d_2})$, and therefore the colon $(x^{d_1}, y^{d_2}):(x+y)^a$ is the whole of $P$ for $a\ge d_1+d_2-1$. (This is why we ask for $a\leq d_{1}+d_{2}-2$ in the theorem.)

Our approach is outlined as follows: we will collect some preliminary observations concerning the degrees of the colon ideal's generators. After a key lemma (see \autoref{in colon ideal with right degrees implies generation}), we prove a version of the main theorem in which $d_{1}=d_{2}$. Finally, we develop a connection between the proposed generators (in the $d_{1}=d_{2}$ case) and their partial derivatives to pave the way for proof of the main theorem.

First, we explain why the colon ideal we study has two generators and derive a relation which the generator degrees satisfy. 

\begin{obs}\label{colon generators relation}
Let $P=\F[x,y]$ where $\F$ is a field of characteristic zero. Consider $(x^{d_{1}},y^{d_{2}}):(x+y)^{a}$ where $2\leq d_{1},d_{2}$ and $1\leq a\leq d_{1}+d_{2}-2$. This ideal has two generators $Q_{1}$ and $Q_{2}$, and these generators satisfy the relation
\begin{equation*}
    \deg(Q_{1})+\deg(Q_{2})=d_{1}+d_{2}-a.
\end{equation*}
\end{obs}

\begin{proof}
This is just unpacking the Hilbert-Burch Theorem (see \cite{Burch} or Theorem 20.15 from \cite{Eisenbud-95}) in our specific context. Let $I$ be the ideal $\big(x^{d_{1}},y^{d_{2}},(x+y)^{a}\big)$ where $2\leq d_{1},d_{2}$ and $1\leq a\leq d_{1}+d_{2}-2$. We have a free resolution of $P/I$:
\begin{equation*}
    0\ \longrightarrow\ P(-e_{1})\oplus P(-e_{2})\ \longrightarrow\ P(-d_{1})\oplus P(-d_{2})\oplus P(-a)\ \longrightarrow\ P\ \longrightarrow\ P/I\ \longrightarrow\ 0.
\end{equation*}
Here, $e_{1}$ and $e_{2}$ are the degrees of a homogeneous basis for
the first syzygy module. We can read the Hilbert series of $P/I$ from the resolution:
\begin{equation*}
    \text{HS}_{P/I}(\tau)=\frac{N(\tau)}{(1-\tau)^{2}}\quad \text{where}\quad N(\tau)=1-(\tau^{d_{1}}+\tau^{d_{2}}+\tau^{a})+(\tau^{e_{1}}+\tau^{e_{2}}).
\end{equation*}
However, $P/I$ has finite length (since $I$ contains a monomial complete intersection), so $\text{HS}_{P/I}(\tau)$ is a polynomial. This means $N(\tau)$ must have a zero of order at least $2$ at $\tau=1$, which is to say $N(1)=0$ and $N'(1)=0$. The latter fact implies
\begin{equation}\label{shift relation}
    -d_{1}-d_{2}-a+e_{1}+e_{2}=0\quad\implies\quad e_{1}+e_{2}=d_{1}+d_{2}+a.
\end{equation}
\indent Now a homogeneous relation of degree $e_{i}$ has the form $A_{i}x^{d_{1}}+B_{i}y^{d_{2}}+Q_{i}(x+y)^{a}=0$ for $i=1,2$; in particular, we have $\deg(Q_{i})=e_{i}-a$. Moreover, the third components $Q_{1}$ and $Q_{2}$ of these syzygies generate $(x^{d_1},y^{d_2}):(x+y)^a$, as an arbitrary polynomial $f$ belongs to $(x^{d_1},y^{d_2}):(x+y)^a$ precisely when there exist $A,B\in P$ such that $Ax^{d_1}+By^{d_2}+f(x+y)^a=0$, i.e. precisely when $f$ occurs as the third component of a syzygy on $x^{d_1},y^{d_2},(x+y)^a$. Thus, after summing $\deg(Q_{1})$ and $\deg(Q_{2})$, an appeal to \eqref{shift relation} brings us to the asserted equality:
\begin{equation*}
    \deg(Q_{1})+\deg(Q_{2})=e_{1}+e_{2}-2a\quad\implies\quad \deg(Q_{1})+\deg(Q_{2})=d_{1}+d_{2}-a.
\end{equation*}
\end{proof}

For the next lemma, recall that the kernel of the endomorphism $\times(x+y)^{a}$ on the quotient $P/(x^{d_{1}},y^{d_{2}})$ is the image of the colon ideal $(x^{d_{1}},y^{d_{2}}):(x+y)^{a}$. Thus, finding the first degree for which injectivity of $\times (x+y)^{a}$ fails is equivalent to knowing the first degree in which a generator of the colon ideal lying outside $(x^{d_1},y^{d_2})$ will be found.

\begin{lem}\label{injectivity failure}
Form the quotient $R=P/(x^{d_{1}},y^{d_{2}})$ where $2\leq d_{1}\leq d_{2}$, and set $k=d_{2}-d_{1}$. Let $\ell$ be a general linear form, and for any fixed integer $a$ satisfying $1\leq a\leq d_{1}+d_{2}-2$ let $\times\ell^{a}:R_{\tau}\rightarrow R_{\tau+a}$ be the map induced by ``multiplication by $\ell^{a}$.'' The first degree $\tau$ where this map fails to be injective is given by
\begin{equation*}
    \tau=\begin{cases}
        d_{2}-a & \text{when}\ a\leq k, \\
        d_{2}-\left\lceil\frac{a+k}{2}\right\rceil & \text{when}\ a\geq k+1.
    \end{cases}
\end{equation*}
\end{lem}

\begin{proof}
Note that $2d_{2}-k-2$ is the socle degree of $R$. As $R$ is a complete intersection, Theorem 1 of \cite{Rei-Rob-Roi} tells us the Hilbert function $\mathrm{HF}_{R}(\tau)$ is symmetric about half the socle degree and moreover that $\mathrm{HF}_{R}(\tau)$ strictly increases from $\tau=0$ to $\tau=d_{2}-k-1$, remains constant from $\tau=d_{2}-k-1$ to $\tau=d_{2}-1$, and strictly decreases from $\tau=d_{2}$ to $\tau=2d_{2}-k-2$. We pair these opening remarks with the classic result of \cite{Stanley-80} (see also Proposition 4.4 in \cite{Har-Mig-Nag-Wat}), which is that $R$ enjoys SLP. The upshot is that $\times\ell^{a}:R_{\tau}\rightarrow R_{\tau+a}$ fails to be injective precisely when 
\begin{equation}\label{HF desire}
    \text{HF}_{R}(\tau)>\text{HF}_{R}(\tau+a).
\end{equation}
Our aim is to find the least nonnegative $\tau$ for which \eqref{HF desire} is satisfied. Notice that $\tau\leq d_{2}-a-1$ will never meet this condition. Indeed, $\tau\leq d_{2}-a-1$ is equivalent to $\tau+a\leq d_{2}-1$, and the Hilbert function is nondecreasing up to the latter degree. Thus, $\tau\geq d_{2}-a$ is necessary for \eqref{HF desire} to hold.

Suppose here that $a\leq k$. This setup implies $d_{2}-k-1<d_{2}-a\leq d_{2}-1$, and using the description of $\text{HF}_{R}(\tau)$ from the first paragraph we obtain
\begin{equation*}
    \text{HF}_{R}(d_{2}-k-1)=\text{HF}_{R}(d_{2}-a)=\text{HF}_{R}(d_{2}-1)>\text{HF}_{R}(d_{2}).
\end{equation*}
This shows that the lower bound established in the previous paragraph actually satisfies \eqref{HF desire}, so $\tau= d_{2}-a$ is the first degree where injectivity fails when $a\leq k$.

Now suppose that $a\geq k+1$. This setup implies $d_{2}-a-1\leq d_{2}-k-1$, meaning that (by exploiting the symmetry of the Hilbert function) the first degree where injectivity fails can be obtained by finding the least $\tau$ such that $\tau+(\tau+a)$ exceeds the socle degree:
\begin{equation*}
    2\tau+a>2d_{2}-k-2\quad\iff\quad \tau>d_{2}-\left\lceil\frac{a+k}{2}\right\rceil-1.
\end{equation*}
This immediately yields $\tau=d_{2}-\left\lceil\frac{a+k}{2}\right\rceil$ as the first degree for which injectivity fails when $a\geq k+1$. (Observe that this value of $\tau$ is positive for $a\leq d_{1}+d_{2}-2$.)
\end{proof}

Pairing the preceding results gives us the degrees of the generators for our colon ideal.

\begin{obs}\label{colon generator degrees}
Let $Q_{1}$ and $Q_{2}$ generate $(x^{d_{1}},y^{d_{2}}):(x+y)^{a}$ where $2\leq d_{1}\leq d_{2}$ and $1\leq a\leq d_{1}+d_{2}-2$. Set $k=d_{2}-d_{1}$.
\begin{enumerate}
    \item If $1\leq a\leq k$, then $\deg(Q_{1})=d_{1}$ and $\deg(Q_{2})=d_{2}-a$.
    \item If $k+1\leq a\leq d_{1}+d_{2}-2$, then $\deg(Q_{1})=d_{2}-\left\lceil\frac{a+k}{2}\right\rceil$ and $\deg(Q_{2})=d_{2}-\left\lfloor\frac{a+k}{2}\right\rfloor$.
\end{enumerate}
\end{obs}

\begin{proof}
Let $Q_{1}$ and $Q_{2}$ generate $(x^{d_{1}},y^{d_{2}}):(x+y)^{a}$. Since $k=d_{2}-d_{1}$, the relation from \autoref{colon generators relation} reads in this setting as
\begin{equation}\label{degrees of gens condition}
    \deg(Q_{1})+\deg(Q_{2})=2d_{2}-(a+k).
\end{equation}
Suppose $1\leq a\leq k$ holds. By \autoref{injectivity failure}, at least one of the generators of $(x^{d_{1}},y^{d_{2}}):(x+y)^{a}$, say $Q_{2}$, has
\begin{equation*}
    \deg(Q_{2})=d_{2}-a.
\end{equation*}
In view of \eqref{degrees of gens condition}, we then obtain
\begin{equation*}
    \deg(Q_{1})+(d_{2}-a)=2d_{2}-(a+k)\quad\implies\quad \deg(Q_{1})=d_{2}-k,\ \ \text{i.e.}\ \ \deg(Q_{1})=d_{1}.
\end{equation*}
Now suppose $k+1\leq a\leq d_{1}+d_{2}-2$ holds. \autoref{injectivity failure} says one of the generators of $(x^{d_{1}},y^{d_{2}}):(x+y)^{a}$, say $Q_{1}$, satisfies
\begin{equation*}
    \deg(Q_{1})=d_{2}-\left\lceil\frac{a+k}{2}\right\rceil.
\end{equation*}
In view of \eqref{degrees of gens condition}, we then obtain
\begin{equation*}
    \left(d_{2}-\left\lceil\frac{a+k}{2}\right\rceil\right)+\deg(Q_{2})=2d_{2}-(a+k)\quad\implies\quad \deg(Q_{2})=d_{2}-\left\lfloor\frac{a+k}{2}\right\rfloor.
\end{equation*}
\end{proof}

The bulk of the proof of \autoref{colon generators when equigenerated} is dedicated to showing that our proposed generators belong to the colon ideal at all, and the following lemma tells us that their membership in the colon ideal is enough.

\begin{lem}\label{in colon ideal with right degrees implies generation}
Retain the setup of \autoref{colon generator degrees}. If $g,h\in(x^{d_{1}},y^{d_{2}}):(x+y)^{a}$ are homogeneous forms satisfying $\deg(g)=\deg(Q_{1})$ and $\deg(h)=\deg(Q_{2})$ with $g\nmid h$, then in fact $(x^{d_{1}},y^{d_{2}}):(x+y)^{a}=(g,h)$.
\end{lem}

\begin{proof}
We go by cases according to whether $\deg(Q_{1})<\deg(Q_{2})$ or $\deg(Q_{1})=\deg(Q_{2})$. Suppose the former situation holds. Since $g\in (Q_{1},Q_{2})$, there exist elements $\alpha_{1},\alpha_{2}\in P$ such that $g=\alpha_{1}Q_{1}+\alpha_{2}Q_{2}$. However, $\deg(g)=\deg(Q_{1})<\deg(Q_{2})$ then implies $\alpha_{2}=0$ and $\alpha_{1}\in\F$, i.e. $g$ is a scalar multiple of $Q_{1}$. Hence, we may replace $Q_{1}$ by $g$ and generate the same ideal. Next, $h\in (g,Q_{2})$ means there exist $\beta_{1},\beta_{2}\in P$ such that $h=\beta_{1}g+\beta_{2}Q_{2}$. We have $\deg(h)=\deg(Q_{2})$, implying that $\beta_{2}\in\F$. In fact, $\beta_{2}\in\F^{\times}$ because otherwise we obtain $h=\beta_{1}g$, a contradiction since $g$ does not divide $h$. All this leads to the following realization:
\begin{equation*}
    h=\beta_{1}g+\beta_{2}Q_{2}\quad\implies\quad Q_{2}=(-\beta_{2})^{-1}\cdot (\beta_{1}g-h)\quad\implies\quad Q_{2}\in (g,h).
\end{equation*}
This completes the proof when $\deg(Q_{1})<\deg(Q_{2})$, so assume now that their degrees are equal. Observe that $g,h\in (Q_{1},Q_{2})$ means there exist elements $\gamma_{1},\gamma_{2},\delta_{1},\delta_{2}\in P$ such that
\begin{equation*}
    g=\gamma_{1}Q_{1}+\gamma_{2}Q_{2}\quad\text{and}\quad h=\delta_{1}Q_{1}+\delta_{2}Q_{2}\qquad\implies\qquad \begin{bmatrix}
        \gamma_{1} & \gamma_{2} \\
        \delta_{1} & \delta_{2}
    \end{bmatrix}\begin{bmatrix}
        Q_{1} \\ Q_{2}
    \end{bmatrix}=\begin{bmatrix}
        g \\ h
    \end{bmatrix}.
\end{equation*}
In this case, notice that
\begin{equation*}
    \deg(g)=\deg(Q_{1})=\deg(Q_{2})=\deg(h)\quad\implies\quad \gamma_{1},\gamma_{2},\delta_{1},\delta_{2}\in\F.
\end{equation*}
We claim that the matrix containing these field entries is invertible, for suppose instead that it is not. Then its rows are $P$-linearly dependent, so there exists $c\in P$ (actually $c\in\F$) such that $\delta_{1}=c\gamma_{1}$ and $\delta_{2}=c\gamma_{2}$. But this yields
\begin{equation*}
    h=c\gamma_{1}Q_{1}+c\gamma_{2}Q_{2}\quad\implies\quad h=c(\gamma_{1}Q_{1}+\gamma_{2}Q_{2})\quad\implies\quad h=c\cdot g.
\end{equation*}
In other words, we have $h$ is a scalar multiple of $g$, contrary to the assumption $g\nmid h$. Therefore, the rows are linearly independent and so the matrix is invertible. We thence obtain
\begin{equation*}
    \left(\begin{bmatrix}
        \gamma_{1} & \gamma_{2} \\
        \delta_{1} & \delta_{2}
    \end{bmatrix}\right)^{-1}\begin{bmatrix}
        g \\ h
    \end{bmatrix}=\begin{bmatrix}
        Q_{1} \\ Q_{2}
    \end{bmatrix}\quad\implies\quad Q_{1},Q_{2}\in(g,h).
\end{equation*}
This establishes the lemma when $\deg(Q_{1})=\deg(Q_{2})$.
\end{proof}

Having dealt with the preparatory results, we are ready to provide the generators of $(x^{d_{1}},y^{d_{2}}):(x+y)^{a}$ when $d_{1}=d_{2}$; they arise by taking $n=0$ in \autoref{formulas}.

\begin{thrm}\label{colon generators when equigenerated}
The colon ideal $(x^{d},y^{d}):(x+y)^{a}$ where $d\geq 2$ and $1\leq a\leq 2d-2$ has the following generators in $P$:
\begin{enumerate}[label=(\arabic*)]
    \item When $a$ is odd, the generators are $F_{1,d,a,0}$ and $F_{2,d,a,0}$.
    \item When $a$ is even, the generators are $G_{1,d, a,0}$ and $G_{2,d,a,0}$.
\end{enumerate}
\end{thrm}

Note that the degrees of the generators in this theorem align with part (2) of \autoref{colon generator degrees}. We remark that each pair of forms described in the theorem are not ``redundant'' in the sense that $(F_{1,d,a,0})\not\subseteq (F_{2,d,a,0})$ and vice versa; this is easily seen by the fact that $F_{2,d,a,0}$ has an irreducible factor of $y$ while $F_{1,d,a,0}$ is not a multiple of $y$. (Note that $P$ is a unique factorization domain.) The same remark holds for $G_{1,d,a,0}$ and $G_{2,d,a,0}$ with the obvious necessary modifications.

To simplify the presentation of our proof, we will suppress the subscripts $d$, $a$, and $0$ on each generator and simply write $F_{1}$, $F_{2}$, $G_{1}$, and $G_{2}$. The notation $\text{cff}_{f}(m)$ will stand for the coefficient of the monomial $m$ in the polynomial $f$. We also point out that the proof, while quite lengthy, is really the same argument carried out four times: we make an ansatz of weights to write each generator as a $P$-linear combination of two others and then show these weights are indeed correct.

\begin{proof}
Fix any $d\geq 2$. Our proof goes by induction on $a$, and we split into cases according to the parity of $a$. Take $a=1$ as the base case for odd $a$. To see that $F_{1}\in (x^{d},y^{d}):(x+y)$, we compute the coefficient of $x^{d-i}y^{i}$ in $(x+y)F_{1}$. This coefficient comes from the $i^{\text{th}}$ and $(i-1)^{\text{st}}$ terms in the expansion of $F_{1}$:
\begin{equation*}
    \text{cff}_{(x+y)F_{1}}(x^{d-i}y^{i})=(-1)^{i}+(-1)^{i-1}\ \implies\ \text{cff}_{(x+y)F_{1}}(x^{d-i}y^{i})=0\ \text{for all}\ i\in\{1,\dots,d-1\}.
\end{equation*}
We then conclude
\begin{equation*}
    (x+y)F_{1}=x^{d}+(-1)^{d-1}y^{d}\quad\implies\quad F_{1}\in (x^{d},y^{d}):(x+y).
\end{equation*}
For $F_{2}$, recall that we specially defined $F_{2}=(-1)^{d-2}(d-1)y^{d}$ when $a=1$ (and $n=0$). As a multiple of $y^{d}$, the desired membership $F_{2}\in (x^{d},y^{d}):(x+y)$ is obvious.

Now take $a=2$ as the base case for even $a$. To see that $G_{1}\in(x^{d},y^{d}):(x+y)^{2}$, we compute the coefficient of $x^{d-i}y^{i+1}$ in $(x+y)^{2}G_{1}$. In view of $(x+y)^{2}=x^{2}+2xy+y^{2}$, this coefficient comes from the $i^{\text{th}}$, $(i-1)^{\text{st}}$, and $(i-2)^{\text{nd}}$ terms in the expansion of $G_{1}$:
\begin{align*}
    & \text{cff}_{(x+y)^{2}G_{1}}(x^{d-i}y^{i+1})=(-1)^{i}(i+1)+2(-1)^{i-1}(i)+(-1)^{i-2}(i-1) \\
    \implies\quad & \text{cff}_{(x+y)^{2}G_{1}}(x^{d-i}y^{i+1})=0\ \text{for all}\ i\in\{1,\dots,d-2\}.
\end{align*}
From here, we see that
\begin{equation*}
    (x+y)^{2}G_{1}=x^{d}y+\underbrace{2(-1)^{d-2}(d-1)xy^{d}+(-1)^{d-3}(d-2)xy^{d}+(-1)^{d-2}(d-1)y^{d+1}}_{\text{multiple of}\ y^{d}}.
\end{equation*}
The preceding line says $G_{1}\in (x^{d},y^{d}):(x+y)^{2}$ holds. We complete this process for $G_{2}$ as well, this time computing the coefficient of $x^{d+1-i}y^{i}$. Once again, we look at the $i^{\text{th}}$, $(i-1)^{\text{st}}$, and $(i-2)^{\text{nd}}$ terms in the expansion of $G_{2}$:
\begin{align*}
    & \text{cff}_{(x+y)^{2}G_{2}}(x^{d+1-i}y^{i})=(-1)^{i}(i-1)+2(-1)^{i-1}(i-2)+(-1)^{i-2}(i-3) \\
    \implies\quad & \text{cff}_{(x+y)^{2}G_{2}}(x^{d+1-i}y^{i})=0\ \text{for all}\ i\in\{2,\dots,d-2\}.
\end{align*}
From here, we see that
\small
\begin{equation*}
    (x+y)^{2}G_{2}=\underbrace{-x^{d+1}-2x^{d}y}_{\text{multiple of}\ x^{d}}+\underbrace{2(-1)^{d-1}(d-2)xy^{d}+(-1)^{d-2}(d-3)xy^{d}+(-1)^{d-1}(d-2)y^{d+1}}_{\text{multiple of}\ y^{d}}.
\end{equation*}
\normalsize
This says that $G_{2}\in(x^{d},y^{d}):(x+y)^{2}$ holds, completing the proof of the base cases.

Proceeding to the inductive step, suppose first $a$ is odd. Our inductive hypothesis is that $(x^{d},y^{d}):(x+y)^{a}$ is generated by $F_{1}$ and $F_{2}$, and we must show that $(x^{d},y^{d}):(x+y)^{a+1}$ is generated by $G_{1}$ and $G_{2}$. We prove that $(x+y)G_{1}, (x+y)G_{2}\in (x^{d},y^{d}):(x+y)^a$ by showing that they are linear combinations of $F_{1}$ and $F_{2}$, and therefore $G_1, G_2\in (x^{d},y^{d}):(x+y)^{a+1}$. Then \autoref{in colon ideal with right degrees implies generation} assures us that $G_{1}$ and $G_{2}$ generate the colon ideal.

Define $b=\frac{a-1}{2}$. Setting
\begin{equation}\label{G1 goal}
    \alpha=d-1,\ \beta=d-b-1,\ \text{and}\ \gamma=d+b,\ \ \text{we claim that}\ \ \alpha\cdot(x+y)G_{1}=\beta\cdot yF_{1}+\gamma\cdot F_{2}.
\end{equation}
To see this, we compute the coefficient of $\displaystyle x^{d-b-i}y^{i}$ on both sides. For $(x+y)G_{1}$, such a monomial is obtained via the $(i-1)^{\text{st}}$ and $(i-2)^{\text{nd}}$ terms from the summation in $G_{1}$. It is
\begin{align*}
    \text{cff}_{(x+y)G_{1}}(x^{d-b-i}y^{i}) & =(-1)^{i-1}\binom{d-2-(i-1)}{b}\binom{(b+1)+(i-1)}{b+1} \\
    & \quad\quad +(-1)^{i-2}\binom{d-2-(i-2)}{b}\binom{(b+1)+(i-2)}{b+1} \\
    & = (-1)^{i-1}\binom{d-1-i}{b}\binom{b+i}{b+1}+(-1)^{i}\binom{d-i}{b}\binom{b-1+i}{b+1}. \numberthis\label{cff-(x+y)G1}
\end{align*}
To compute the corresponding coefficient in $yF_{1}$, note that the multiplication by $y$ indicates we must look at the $(i-1)^{\text{st}}$ term in the expansion of $F_{1}$:
\begin{align*}
    \text{cff}_{yF_{1}}(x^{d-b-i}y^{i}) & =(-1)^{i-1}\binom{d-1-(i-1)}{b}\binom{b+(i-1)}{b} \\ 
    & =(-1)^{i-1}\binom{d-i}{b}\binom{b-1+i}{b}. \numberthis\label{cff-yF1}
\end{align*}
Finally, recall that $F_{2}$ is a multiple of $y^{2}$, so the corresponding coefficient comes from the $(i-2)$ term in the expansion of $F_{2}$:
\begin{align*}
    \text{cff}_{F_{2}}(x^{d-b-i}y^{i}) & =(-1)^{i-2}\binom{d-3-(i-2)}{b-1}\binom{(b+1)+(i-2)}{b+1} \\
    & =(-1)^{i}\binom{d-1-i}{b-1}\binom{b-1+i}{b+1}. \numberthis\label{cff-F2}
\end{align*}
Our task is to show that $\alpha\cdot\text{cff}_{(x+y)G_{1}}(x^{d-b-i}y^{i})=\beta\cdot\text{cff}_{yF_{1}}(x^{d-b-i}y^{i})+\gamma\cdot\text{cff}_{F_{2}}(x^{d-b-i}y^{i})$. We will assume here that $i$ is odd, the case ``$i$ is even'' a simple change of sign in what follows. Define
\begin{equation*}
    A=\binom{d-i}{b}\qquad\text{and}\qquad B=\binom{b-1+i}{b+1}.
\end{equation*}
We rewrite \eqref{cff-(x+y)G1} in terms of $A$ and $B$:
\begin{align*}
    \alpha\cdot\text{cff}_{(x+y)G_{1}}(x^{d-b-i}y^{i}) & =(d-1)\left[\frac{d-b-i}{d-i}A\cdot\frac{b+i}{i-1}B-AB\right] \\
    & =AB\left[(d-1)\left(\frac{(d-b-i)(b+i)}{(d-i)(i-1)}-1\right)\right].
\end{align*}
Carry out these proceedings for the sum of \eqref{cff-yF1} and \eqref{cff-F2} as well:
\begin{align*}
    \beta\cdot\text{cff}_{yF_{1}}(x^{d-b-i}y^{i})+\gamma\cdot\text{cff}_{F_{2}}(x^{d-b-i}y^{i}) & =(d-b-1)\left[A\cdot\frac{b+1}{i-1}B\right]-(d+b)\left[\frac{b}{d-i}A\cdot B\right] \\
    & =AB\left[\frac{(d-b-1)(b+1)}{i-1}-\frac{(d+b)(b)}{d-i}\right].
\end{align*}
It remains to verify that the bracketed expressions on the right-hand sides above are equal. After clearing denominators via multiplication by $(d-i)(i-1)$, expanding those expressions reveals that they are both equal to $d^{2}b-db^{2}-2dbi+d^{2}-db+b^{2}-di+2bi-d+i$. This proves $\eqref{G1 goal}$.

Keeping $b=\frac{a-1}{2}$, we next claim that
\begin{equation*}
    \alpha=1-d,\ \ \beta=d-1,\ \ \gamma=(d-1)+(b+1)(d-b-1),\ \ \text{and}\ \ \delta=(b+1)(d+b)
\end{equation*}
\begin{equation}\label{G2 goal}
    \text{will satisfy}\qquad \alpha\cdot (x+y)G_{2}=(\beta x+\gamma y)\cdot F_{1}+\delta\cdot F_{2}.
\end{equation}
To see this, we compute the coefficient of $x^{d-b-i}y^{i}$ on both sides. For $(x+y)G_{2}$, such a monomial is obtained via the $i^{\text{th}}$ and $(i-1)^{\text{st}}$ terms from the summation in $G_{2}$. It is
\begin{align*}
    \text{cff}_{(x+y)G_{2}}(x^{d-b-i}y^{i}) & = (-1)^{i}\binom{d-1-i}{b}\binom{b+i}{b}(i-1) \\ 
    & \quad\quad +(-1)^{i-1}\binom{d-1-(i-1)}{b}\binom{b+(i-1)}{b}\big((i-1)-1\big) \\
    & =(-1)^{i}\binom{d-1-i}{b}\binom{b+i}{b}(i-1) \\
    & \quad\quad +(-1)^{i-1}\binom{d-i}{b}\binom{b-1+i}{b}(i-2). \numberthis\label{cff-(x+y)G2}
\end{align*}
To compute the corresponding coefficient of $(\beta x+\gamma y)F_{1}$, we look at the $i^{\text{th}}$ and $(i-1)^{\text{st}}$ terms in the expansion of $F_{1}$. As the latter was computed in \eqref{cff-yF1}, we only provide the former below:
\begin{equation}\label{cff-xF1}
    \text{cff}_{xF_{1}}(x^{d-b-i}y^{i})=(-1)^{i}\binom{d-1-i}{b}\binom{b+i}{b}.
\end{equation}
Finally, recall that $F_{2}$ is a multiple of $y^{2}$, so the corresponding coefficient comes from the $(i-2)$ term in the expansion of $F_{2}$. This was found in \eqref{cff-F2}.

Our task is to show that $\alpha\cdot\text{cff}_{(x+y)G_{2}}=\beta\cdot\text{cff}_{xF_{1}}+\gamma\cdot\text{cff}_{yF_{1}}+\delta\cdot\text{cff}_{F_{2}}$. We will assume here that $i$ is even, the case ``$i$ is odd'' a simple change of sign in what follows. Define
\begin{equation*}
    A=\binom{d-1-i}{b}\qquad\text{and}\qquad B=\binom{b-1+i}{b}.
\end{equation*}
We rewrite \eqref{cff-(x+y)G2} in terms of $A$ and $B$:
\begin{align*}
    \alpha\cdot\text{cff}_{(x+y)G_{2}}(x^{d-b-i}y^{i}) & =(1-d)\left[A\cdot\frac{b+i}{i}B\cdot(i-1)-\frac{d-i}{d-b-i}A\cdot B\cdot(i-2)\right] \\
    & =AB\left[(1-d)\left(\frac{(b+i)(i-1)}{i}-\frac{(d-i)(i-2)}{d-b-i}\right)\right]. \numberthis \label{G2-gen cff}
\end{align*}
Next, do this for the sum of coefficients computed in \eqref{cff-xF1} and \eqref{cff-yF1}:
\small
\begin{align*}
    \beta\cdot\text{cff}_{xF_{1}}(x^{d-b-i}y^{i})+\gamma\cdot\text{cff}_{yF_{1}}(x^{d-b-i}y^{i}) & =(d-1)\left[A\cdot\frac{b+i}{i}B\right] \\
    & \quad\quad -((d-1)+(b+1)(d-b-1))\left[\frac{d-i}{d-b-i}A\cdot B\right] \\
    & =AB\left[(d-1)\left(\frac{b+i}{i}-\frac{d-i}{d-b-i}\right)\right. \\
    & \qquad\qquad\qquad \left.-\frac{(b+1)(d-b-1)(d-i)}{d-b-i}\right].
\end{align*}
\normalsize
Finally, give this treatment to the coefficient from \eqref{cff-F2}:
\begin{align*}
    \delta\cdot\text{cff}_{F_{2}}(x^{d-b-i}y^{i}) & =(b+1)(d+b)\left[\frac{b}{d-b-i}A\cdot\frac{i-1}{b+1}B\right] \\
    & =AB\left[\frac{(d+b)(b)(i-1)}{d-b-i}\right].
\end{align*}
Piecing this information together, we have
\small
\begin{align*}
    \beta \cdot\text{cff}_{xF_{1}}(x^{d-b-i}y^{i})+\gamma \cdot\text{cff}_{yF_{1}}(x^{d-b-i}y^{i}) & +\delta \cdot\text{cff}_{F_{2}}(x^{d-b-i}y^{i}) \\
    & =AB\left[(d-1)\left(\frac{b+i}{i}-\frac{d-i}{d-b-i}\right)\right. \\
    & \qquad\qquad \left.-\,\frac{(b+1)(d-b-1)(d-i)}{d-b-i}+\frac{(d+b)(b)(i-1)}{d-b-i}\right].
\end{align*}
\normalsize
It remains to verify that the bracketed expression on the right-hand side above is equal to that in \eqref{G2-gen cff}. After clearing denominators via multiplication by $i(d-b-i)$, expanding those expressions reveals that they are both equal to $-d^{2}bi+db^{2}i+2dbi^{2}+d^{2}b-db^{2}-d^{2}i-dbi-b^{2}i+di^{2}-2bi^{2}-db+b^{2}+di+2bi-i^{2}$. This proves \eqref{G2 goal} and completes the proof when $a$ is odd.

Suppose now that $a$ is even. Our inductive hypothesis is that $(x^{d},y^{d}):(x+y)^{a}$ is generated by $G_{1}$ and $G_{2}$, and we must show that $(x^{d},y^{d}):(x+y)^{a+1}$ is generated by $F_{1}$ and $F_{2}$. We will show that $(x+y)F_{1}$ and $(x+y)F_{2}$ are linear combinations of $G_{1}$ and $G_{2}$, which implies $F_1, F_2\in (x^{d},y^{d}):(x+y)^{a+1}$. Then \autoref{in colon ideal with right degrees implies generation} will imply $F_{1}$ and $F_{2}$ generate the colon ideal.

Define $b=\frac{a}{2}$. Setting
\begin{equation}\label{F1 goal}
    \alpha=b,\ \beta=b(b+2-d),\ \text{and}\ \gamma=b-d,\ \ \text{we claim that}\ \ \alpha\cdot(x+y)F_{1}=\beta\cdot G_{1}+\gamma\cdot G_{2}.
\end{equation}
To see this, we compute the coefficient of $x^{d-b-i}y^{i}$ on both sides. For $(x+y)F_{1}$, such a monomial is obtained via the $i^{\text{th}}$ and $(i-1)^{\text{st}}$ terms in the expansion of $F_{1}$. These are given by \eqref{cff-xF1} and \eqref{cff-yF1}, respectively.

Recall that $G_{1}$ is a multiple of $y$, so we look at the $(i-1)^{\text{st}}$ term in the expansion of $G_{1}$:
\begin{align*}
    \text{cff}_{G_{1}}(x^{d-b-i}y^{i}) & =(-1)^{i-1}\binom{d-2-(i-1)}{b-1}\binom{b+(i-1)}{b} \\
    & =(-1)^{i-1}\binom{d-1-i}{b-1}\binom{b-1+i}{b}. \numberthis \label{cff-G1}
\end{align*}
For $G_{2}$, we need only look at the $i^{\text{th}}$ term in the summation:
\begin{equation}
    \text{cff}_{G_{2}}(x^{d-b-i}y^{i})=(-1)^{i}\binom{d-1-i}{b-1}\binom{b-1+i}{b-1}(i-1) \numberthis\label{cff-G2}.
\end{equation}
Our task is to show that $\alpha\cdot\text{cff}_{(x+y)F_{1}}(x^{d-b-i}y^{i})=\beta\cdot\text{cff}_{G_{1}}(x^{d-b-i}y^{i})+\gamma\cdot\text{cff}_{G_{2}}(x^{d-b-i}y^{i})$. We will assume here that $i$ is odd, the case ``$i$ is even'' a simple change of sign in what follows. Define
\begin{equation*}
    A=\binom{d-1-i}{b-1}\qquad\text{and}\qquad B=\binom{b-1+i}{b}.
\end{equation*}
We rewrite the sum of \eqref{cff-xF1} and \eqref{cff-yF1} in terms of $A$ and $B$:
\begin{align*}
    \alpha\cdot\text{cff}_{(x+y)F_{1}}(x^{d-b-i}y^{i}) & =b\left[-\frac{d-b-i}{b}A\cdot\frac{b+i}{i}B+\frac{d-i}{b}A\cdot B\right] \\
    & =AB\left[(d-i)-\frac{(d-b-i)(b+i)}{i}\right].
\end{align*}
Do the same for the sum of \eqref{cff-G1} and \eqref{cff-G2}:
\begin{align*}
    \beta\cdot\text{cff}_{G_{1}}(x^{d-b-i}y^{i})+\gamma\cdot\text{cff}_{G_{2}}(x^{d-b-i}y^{i}) & =b(b+2-d)\big[AB\big]-(b-d)\left[A\cdot\frac{b}{i}(i-1)B\right] \\
    & =AB\left[b\left((b+2-d)-\frac{(b-d)(i-1)}{i}\right)\right].
\end{align*}
It remains to verify that the bracketed expressions on the right-hand sides above are equal. After clearing denominators via multiplication by $i$, expanding those expressions reveals that they are both equal to $-db+b^{2}+2bi$. This proves \eqref{F1 goal}.

We make a minor adjustment to $b$ here. Taking $b=\frac{a-2}{2}$, our last claim is that
\begin{equation*}
    \alpha=-(b+2),\ \ \beta=d-1,\ \ \gamma=(b+2)(d-b-2),\ \ \text{and}\ \ \delta=d-b-1
\end{equation*}
\begin{equation}\label{F2 goal}
    \text{will satisfy}\qquad \alpha(x+y)F_{2}=(\beta x+\gamma y)\cdot G_{1}+\delta\cdot yG_{2}.
\end{equation}
To see this, we compute the coefficient of $x^{d-b-i}y^{i}$ on both sides. For $(x+y)F_{2}$, such a monomial is obtained via the $(i-2)^{\text{nd}}$ and $(i-3)^{\text{rd}}$ terms in the expansion of $F_{2}$. It is
\begin{align*}
    \text{cff}_{(x+y)F_{2}}(x^{d-b-i}y^{i}) & =(-1)^{i-2}\binom{d-3-(i-2)}{b}\binom{b+2+(i-2)}{b+2} \\ 
    & \quad\quad +(-1)^{i-3}\binom{d-3-(i-3)}{b}\binom{b+2+(i-3)}{b+2}\\
    & =(-1)^{i}\binom{d-1-i}{b}\binom{b+i}{b+2}+(-1)^{i-1}\binom{d-i}{b}\binom{b-1+i}{b+2}. \numberthis\label{cff-(x+y)F2}
\end{align*}
Recall that $G_{1}$ is a multiple of $y$, so we look at the $(i-1)^{\text{st}}$ and $(i-2)^{\text{nd}}$ terms in the expansion of $G_{1}$:
\small
\begin{equation}
    \begin{aligned}
        \text{cff}_{xG_{1}}(x^{d-b-i}y^{i}) & =(-1)^{i-1}\binom{d-1-i}{b}\binom{b+i}{b+1}\qquad \text{and} \\
        \text{cff}_{yG_{1}}(x^{d-b-i}y^{i}) & =(-1)^{i-2}\binom{d-i}{b}\binom{b-1+i}{b+1}.
    \end{aligned}\label{cff-(bx+cy)G1}
\end{equation}
\normalsize
Finally, we need the $(i-1)^{\text{st}}$ term from the expansion of $G_{2}$:
\begin{align*}
    \text{cff}_{yG_{2}}(x^{d-b-i}y^{i}) & =(-1)^{i-1}\binom{d-1-(i-1)}{b}\binom{b+(i-1)}{b}\big((i-1)-1\big) \\
    & =(-1)^{i-1}\binom{d-i}{b}\binom{b-1+i}{b}(i-2).\numberthis\label{cff-yG2}
\end{align*}
We will show that $\alpha\cdot\text{cff}_{(x+y)F_{2}}(x^{d-b-i}y^{i})=\beta\cdot\text{cff}_{xG_{1}}(x^{d-b-i}y^{i})+\gamma\cdot\text{cff}_{yG_{1}}+\delta\cdot\text{cff}_{yG_{2}}(x^{d-b-i}y^{i})$. Assume here that $i$ is odd, the case ``$i$ is even'' a simple change of sign in what follows. Define
\begin{equation*}
    A=\binom{d-i}{b}\qquad\text{and}\qquad B=\binom{b-1+i}{b+1}.
\end{equation*}
As before, we rewrite \eqref{cff-(x+y)F2} in terms of $A$ and $B$:
\begin{align*}
    \alpha\cdot\text{cff}_{(x+y)F_{2}}(x^{d-b-i}y^{i}) & =-(b+2)\left[-\frac{d-b-i}{d-i}A\cdot\frac{b+i}{b+2}B+A\cdot\frac{i-2}{b+2}B\right]\\
    & =AB\left[\frac{(d-b-i)(b+i)}{d-i}-(i-2)\right]. \numberthis\label{F2-gen cff}
\end{align*}
Carry this out for the coefficients in \eqref{cff-(bx+cy)G1}: 
\small
\begin{align*}
    \beta\cdot\text{cff}_{xG_{1}}(x^{d-b-i}y^{i})+\gamma\cdot\text{cff}_{yG_{1}}(x^{d-b-i}y^{i}) & = (d-1)\left[\frac{d-b-i}{d-i}A\cdot\frac{b+i}{i-1}B\right] \\
    & \qquad\qquad\qquad\qquad-(b+2)(d-b-2)\big[AB\big] \\[5pt]
    & =AB\left[\frac{(d-1)(d-b-i)(b+i)}{(d-i)(i-1)}-(b+2)(d-b-2)\right],
\end{align*}
\normalsize
as well as the coefficient from \eqref{cff-yG2}:
\begin{align*}
    \delta\cdot\text{cff}_{yG_{2}}(x^{d-b-i}y^{i}) & =(d-b-1)\left[A\cdot\frac{b+1}{i-1}B\cdot (i-2)\right] \\
    & =AB\left[\frac{(d-b-1)(b+1)(i-2)}{i-1}\right].
\end{align*}
Piecing this information together, we have
\small
\begin{align*}
    \beta\cdot\text{cff}_{xG_{1}}(x^{d-b-i}y^{i})+\gamma\cdot\text{cff}_{yG_{1}}(x^{d-b-i}y^{i}) & +\delta\cdot\text{cff}_{yG_{2}}(x^{d-b-i}y^{i}) \\
    & =AB\left[\frac{(d-1)(d-b-i)(b+i)}{(d-i)(i-1)}\right. \\
    & \qquad\qquad -\,\left.(b+2)(d-b-2)+\frac{(d-b-1)(b+1)(i-2)}{i-1}\right].
\end{align*}
\normalsize
It remains to verify that the bracketed expression on the right-hand side above is equal to that in \eqref{F2-gen cff}. After clearing denominators via multiplication by $(d-i)(i-1)$, expanding those expressions reveals that they are both equal to $dbi-b^{2}i-2bi^{2}-db+b^{2}+2di+2bi-2i^{2}-2d+2i$. This proves \eqref{F2 goal} and completes the proof when $a$ is even, establishing the theorem.
\end{proof}

Having dealt with the case where $d_{1}=d_{2}$, we now set our sights on generators of $(x^{d_{1}},y^{d_{2}}):(x+y)^{a}$ when $d_{1}\leq d_{2}$. The tool that will allow us to pass from the case $d_{1}=d_{2}$ to the case $d_{1}\leq d_{2}$ is the partial derivative. We begin by observing that the formulas provided in \autoref{formulas} are scalar multiples of the partial derivatives of $F_{1,d,a,0}$, $F_{2,d,a,0}$, $G_{1,d,a,0}$, and $G_{2,d,a,0}$ with respect to $x$. (Note however that when $a=1$, the identity for $F_{2}$ is ``degenerate'' in the sense that $F_{2,d,1,0}$ has no terms containing $x$, so its partial derivative is $0$. Because of this, we provide a separate ``meaningful'' recursion for the derivative of $F_{2,d,1,n}$ when $n\geq 1$.) We follow the standard convention that empty products are interpreted as $1$.

\begin{obs}\label{partial derivs of generators}
For every $n\geq 0$, the $n^{th}$ partial derivatives of $F_{1,d,a,0}$, $F_{2,d,a,0}$, $G_{1,d,a,0}$, and $G_{2,d,a,0}$ with respect to $x$ satisfy the following: when $a$ is odd,
\begin{equation*}
    \frac{\partial^{n}}{\partial x^{n}}[F_{1,d,a,0}]=\prod_{j=1}^{n}\left[\frac{a-1}{2}+j\right]\cdot F_{1,d,a,n}\quad\text{and}\quad \frac{\partial^{n}}{\partial x^{n}}[F_{2,d,a,0}]=\prod_{j=1}^{n}\left[\frac{a-3}{2}+j\right]\cdot F_{2,d,a,n}\, ;
\end{equation*}
when $a$ is even,
\begin{equation*}
    \frac{\partial^{n}}{\partial x^{n}}[G_{1,d,a,0}]=\prod_{j=1}^{n}\left[\frac{a-2}{2}+j\right]\cdot G_{1,d,a,n}\quad \text{and}\quad \frac{\partial^{n}}{\partial x^{n}}[G_{2,d,a,0}]=\prod_{j=1}^{n}\left[\frac{a-2}{2}+j\right]\cdot G_{2,d,a,n}.
\end{equation*}
In the special case $a=1$, we have the following derivative recursion for every $n\geq 1$:
\begin{equation*}
    \frac{\partial^{n-1}}{\partial x^{n-1}}\big[F_{2,d,1,1}\big]=(n-1)!\cdot F_{2,d,1,n}.
\end{equation*}
\end{obs}

\begin{proof}
It is straightforward to see that the following is $n^{\text{th}}$ partial derivative of $F_{1,d,a,0}$ with respect to $x$:
\begin{equation*}
    \frac{\partial^{n}}{\partial x^{n}}[F_{1,d,a,0}]=\sum_{i=0}^{d-\frac{a+1}{2}-n}(-1)^{i}\binom{d-1-i}{\frac{a-1}{2}}\binom{\frac{a-1}{2}+i}{\frac{a-1}{2}}\left[\prod_{j=0}^{n-1}\left(d-\frac{a+1}{2}-i-j\right)\right] x^{d-\frac{a+1}{2}-i-n}y^{i}.
\end{equation*}
Observe that $d-\frac{a+1}{2}=d-\frac{a-1}{2}-1$. Now define the quantity $C$ as follows:
\begin{align*}
    C & =\binom{d-1-i}{\frac{a-1}{2}}\prod_{j=0}^{n-1}\left(d-\frac{a-1}{2}-i-j-1\right) \\
    & =\frac{(d-1-i)!}{\left(\frac{a-1}{2}\right)!\left(d-\frac{a-1}{2}-i-1\right)!}\cdot\left[d-\frac{a-1}{2}-i-1\right]\cdots\left[d-\frac{a-1}{2}-i-n\right] \\
    & =\frac{(d-1-i)!}{\left(\frac{a-1}{2}\right)!\left(d-\frac{a-1}{2}-i-(n+1)\right)!}.
\end{align*}
When we multiply $C$ by the reciprocal of $\left(\frac{a-1}{2}+1\right)\left(\frac{a-1}{2}+2\right)\cdots\left(\frac{a-1}{2}+n\right)$, we obtain
\begin{equation*}
    C\cdot\frac{1}{\left(\frac{a-1}{2}+1\right)\left(\frac{a-1}{2}+2\right)\cdots\left(\frac{a-1}{2}+n\right)}=\frac{(d-1-i)!}{\left(\frac{a-1}{2}+n\right)!\left(d-\frac{a-1}{2}-i-n-1\right)!}=\binom{d-1-i}{\frac{a-1}{2}+n}.
\end{equation*}
Thus, 
\begin{equation*}
    \frac{\frac{\partial^{n}}{\partial x^{n}}[F_{1,d,a,0}]}{\prod_{j=1}^{n}\left[\frac{a-1}{2}+j\right]}=\sum_{i=0}^{d-\frac{a+1}{2}-n}(-1)^{i}\binom{d-1-i}{\frac{a-1}{2}+n}\binom{\frac{a-1}{2}+i}{\frac{a-1}{2}}x^{d-\frac{a+1}{2}-n-i}y^{i},
\end{equation*}
\begin{equation*}
    \text{which is to say}\qquad \frac{\frac{\partial^{n}}{\partial x^{n}}[F_{1,d,a,0}]}{\prod_{j=1}^{n}\left[\frac{a-1}{2}+j\right]}=F_{1,d,a,n}.
\end{equation*}
Multiplying both sides of the equation above by $\prod_{j=1}^{n}\left[\frac{a-1}{2}+j\right]$ achieves the formula asserted in the observation. In particular, note that $\frac{\partial^{n}}{\partial x^{n}}[F_{1,d,a,0}]$ is a scalar multiple of $F_{1,d,a,n}$. Therefore,
\begin{equation*}
    \left(\frac{\partial^{n}}{\partial x^{n}}[F_{1,d,a,0}]\right)=(F_{1,d,a,n})\ \ \text{as ideals for all}\ n.
\end{equation*}
The formulas for the $n^{\text{th}}$ partial derivatives of $F_{2,d,a,0}$, $G_{1,d,a,0}$, and $G_{2,d,a,0}$ with respect to $x$ are obtained the same way \textit{mutatis mutandis}.

In the special case $a=1$, the $(n-1)^{\text{st}}$ partial derivative of $F_{2,d,1,1}$ for $n\geq 1$ is
\begin{equation*}
    \frac{\partial^{n-1}}{\partial x^{n-1}}\big[F_{2,d,1,1}\big]=y^{2}\sum_{i=0}^{d-2-n}(-1)^{i}(i+1)\left[\prod_{j=0}^{n-2}(d-3-i-j)\right]x^{d-2-n-i}y^{i}.
\end{equation*}
The formula asserted in the observation follows at once after noting
\begin{equation*}
    \prod_{j=0}^{n-2}(d-3-i-j)=(n-1)!\binom{d-3-i}{-1+n}.
\end{equation*}
\end{proof}

We need one more tool, namely that the partial derivative lies in a colon ideal with the exponents suitably adjusted.

\begin{lem}\label{partial deriv belongs to colon ideal}
Consider the ideal $(x^{d_{1}},y^{d_{2}}):(x+y)^{a}$ where $2\leq d_{1}\leq d_{2}$ and $a$ is a nonnegative integer, and let $k=d_{2}-d_{1}$. Assume $a>k$. If $f\in (x^{d_{1}+k},y^{d_{2}}):(x+y)^{a-k}$, then $\frac{\partial^{k}}{\partial x^{k}}[f]\in (x^{d_{1}},y^{d_{2}}):(x+y)^{a}$.
\end{lem}

\begin{proof}
We shall prove something stronger, namely that
\begin{equation}\label{claim}
    \frac{\partial^{j}}{\partial x^{j}}[f]\in (x^{d_{1}+k-j},y^{d_{2}}):(x+y)^{a-k+j}\ \ \text{for all}\ \ 0\leq j\leq k.
\end{equation}
Going by induction on $j$ (up to $k$), there's nothing to prove for the base case $j=0$, so fix any $j\in\{1,\dots,k\}$. Our inductive hypothesis is that $\frac{\partial^{j-1}}{\partial x^{j-1}}[f]\in (x^{d_{1}+k-(j-1)},y^{d_{2}}):(x+y)^{a-k+(j-1)}$. We are therefore justified in writing
\begin{equation*}
    \frac{\partial^{j-1}}{\partial x^{j-1}}[f]\cdot (x+y)^{a-k+j-1}=\widetilde{g}\cdot x^{d_{1}+k-j+1}+\widetilde{h}\cdot y^{d_{2}}\quad\text{for some}\ \widetilde{g},\widetilde{h}\in P.
\end{equation*}
Taking the partial derivative with respect to $x$ on both sides yields
\begin{align*}
    \frac{\partial^{j} f}{\partial x^{j}}(x+y)^{a-k+j-1}+\, & \frac{\partial^{j-1}f}{\partial x^{j-1}}(a-k+j-1)(x+y)^{a-k+j-2} \\
    & =\left(\frac{\partial\widetilde{g}}{\partial x}x^{d_{1}+k-j+1}+\widetilde{g}(d_{1}+k-j+1)x^{d_{1}+k-j}\right)+\frac{\partial\widetilde{h}}{\partial x}y^{d_{2}} \\
    & =\left(\frac{\partial\widetilde{g}}{\partial x}x+\widetilde{g}(d_{1}+k-j+1)\right)x^{d_{1}+k-j}+\left(\frac{\partial\widetilde{h}}{\partial x}\right)y^{d_{2}}.
\end{align*}
The upshot of the preceding is that $\frac{\partial^{j} f}{\partial x^{j}}\cdot (x+y)^{a-k+j-1}+\frac{\partial^{j-1}f}{\partial x^{j-1}}\cdot (a-k+j-1)(x+y)^{a-k+j-2}$ lies in the ideal $(x^{d_{1}+k-j},y^{d_{2}})$. We multiply this expression by $(x+y)$ to reach
\begin{equation*}
    \frac{\partial^{j} f}{\partial x^{j}}\cdot (x+y)^{a-k+j}+\frac{\partial^{j-1}f}{\partial x^{j-1}}\cdot (a-k+j-1)(x+y)^{a-k+j-1}\in(x^{d_{1}+k-j},y^{d_{2}}).
\end{equation*}
The second term above contains $\frac{\partial^{j-1}f}{\partial x^{j-1}}\cdot (x+y)^{a-k+j-1}$, which by induction belongs to $(x^{d_{1}+k-j+1},y^{d_{2}})$. (Note here that $a-k+j-1>0$ because $0\leq k<a$ and $j\geq 1$.) As this ideal is contained in $(x^{d_{1}+k-j},y^{d_{2}})$, it follows immediately that
\begin{equation*}
    \frac{\partial^{j} f}{\partial x^{j}}\cdot (x+y)^{a-k+j}\in (x^{d_{1}+k-j},y^{d_{2}}),\ \text{i.e.}\ \frac{\partial^{j} f}{\partial x^{j}}\in(x^{d_{1}+k-j},y^{d_{2}}):(x+y)^{a-k+j},\ \text{the content of \eqref{claim}}.
\end{equation*}
\end{proof}

Now we are ready to prove the main theorem of this section, whose statement we recall here for the reader's convenience:

\generalgens*

Before embarking on the proof, we mention in passing that instead of $H_{d_{1},a,k}$ as given in part (1) of the theorem we may alternatively use
\begin{equation*}
    H_{d_{1},a,k}'=\sum_{j=0}^{d_2-a} (-1)^j \binom{d_2-j-1}{a-1} x^{d_2-a-j}y^j.
\end{equation*}
The expression $H_{d_{1},a,k}'$ is more akin to the formulas in \autoref{formulas}, but notice that its first terms will be multiples of $x^{d_{1}}$ since $d_{1}=d_{2}-k\leq d_{2}-a$ in part (1). It follows that $H_{d_{1},a,k}'+(-1)^{k-a}H_{d_{1},a,k}\in (x^{d_{1}})$ and therefore $(x^{d_1}, H_{d_{1},a,k})=(x^{d_1}, H_{d_{1},a,k}')$.

To simplify the presentation in our proof of (1), we will suppress the subscripts $d_{1}$ and $k$, simply writing $H_{a}$ when proving the theorem. We also recall that ``$\text{cff}_{f}(m)$'' refers to the coefficient of the monomial $m$ in the polynomial $f$.

\begin{proof}
Assume first that $1\leq a\leq k$. By \autoref{colon generator degrees}, one generator of the colon ideal has degree $d_{1}$ and the other has degree $d_{2}-a$. The element $x^{d_1}$ obviously belongs to the colon ideal and has one of the required generator degrees. Since $x^{d_{1}}\nmid H_{a}$, it remains by \autoref{in colon ideal with right degrees implies generation} to show that \(H_a\in(x^{d_{1}},y^{d_{2}}):(x+y)^{a}\).

Going by induction on $a$, take $a=1$. We show that $(x+y)H_{1}\in (x^{d_{1}},y^{d_{2}})$ holds. To see this, compute the coefficient of $x^{d_{1}-i}y^{k+i}$ in $(x+y)H_{1}$ for $1\leq i\leq d_{1}-1$, which comes from the $i^{\text{th}}$ and $(i-1)^{\text{st}}$ terms in the expansion of $H_{1}$:
\begin{align*}
    & \text{cff}_{(x+y)H_{1}}(x^{d_{1}-i}y^{k+i})=(-1)^{i}\binom{d_{1}-1-i}{0}+(-1)^{i-1}\binom{d_{1}-i}{0} \\
    \implies\quad & \text{cff}_{(x+y)H_{1}}(x^{d_{1}-i}y^{k+i})=0\ \text{for all}\ i\in\{1,\dots,d_{1}-1\}.
\end{align*}
Therefore,
\small
\begin{equation*}
    (x+y)H_{1}=y^{k}\big[x^{d_{1}}+(-1)^{d_{1}-1}y^{d_{1}}\big]=x^{d_{1}}y^{d_{2}-d_{1}}+(-1)^{d_{1}-1}y^{d_{2}}\ \implies\ H_{1}\in (x^{d_{1}},y^{d_{2}}):(x+y).
\end{equation*}
\normalsize

We next take as our inductive hypothesis that $H_{a}\in (x^{d_{1}},y^{d_{2}}):(x+y)^{a}$ for a fixed $a<k$ and endeavor to show that $H_{a+1}\in (x^{d_{1}},y^{d_{2}}):(x+y)^{a+1}$. Observe here that if we can prove that $(x+y)H_{a+1}-H_{a}\in (x^{d_1})$, then it will follow that
\begin{equation}\label{H_{a+1} goal}
    (x+y)^{a+1}H_{a+1}-(x+y)^{a}H_{a}\in (x^{d_1}) 
   \end{equation}
and therefore $ (x+y)^{a+1}H_{a+1}\in (x^{d_{1}},y^{d_{2}})$ as desired. Thus, we set our sights on proving $(x+y)H_{a+1}-H_{a}\in (x^{d_1})$ to complete the inductive step. This comes down to determining the coefficient of $x^{d_{1}-i}y^{k-a+i}$ in $(x+y)H_{a+1}$, which again arises from the and $i^{\text{th}}$ and $(i-1)^{\text{st}}$ terms in the expansion of $H_{a+1}$. With the aid of Pascal's rule, we obtain
\begin{align*}
    \text{cff}_{(x+y)H_{a+1}}(x^{d_{1}-i}y^{k-a+i}) & =(-1)^{i}\binom{d_{1}+a-1-i}{a}+(-1)^{i-1}\binom{d_{1}+a-i}{a} \\
    & =(-1)^{i-1}\binom{d_{1}+a-1-i}{a-1}.
\end{align*}
The upshot of this observation is the realization that
\small
\begin{align*}
    (x+y)H_{a+1} & =y^{k-a}\left[\binom{d_{1}+a-1}{a}x^{d_{1}}+\sum_{i=1}^{d_{1}-1}(-1)^{i-1}\binom{d_{1}+a-1-i}{a-1}x^{d_{1}-i}y^{i}+(-1)^{d_{1}-1}y^{d_{1}}\right] \\
    &=y^{k-a}\left[\binom{d_{1}+a-1}{a}x^{d_{1}}+\sum_{i=0}^{d_{1}-2}(-1)^{i}\binom{d_{1}+a-2-i}{a-1}x^{d_{1}-1-i}y^{i+1}+(-1)^{d_{1}-1}y^{d_{1}}\right].
\end{align*}
\normalsize
Notice that the first term is a multiple of $x^{d_{1}}$, so it lies in $(x^{d_{1}},y^{d_{2}})$. We therefore ignore it and, after incorporating the last term into the summation and extracting a factor of $y$, reach
\begin{equation*}
    (x+y)H_{a+1}\equiv y^{k-a+1}\sum_{i=0}^{d_{1}-1}(-1)^{i}\binom{d_{1}+a-2-i}{a-1}x^{d_{1}-1-i}y^{i}\quad \text{mod}\ (x^{d_{1}}),
\end{equation*}
which says precisely that $(x+y)H_{a+1}\equiv H_{a}\! \mod{(x^{d_{1}})}$. We conclude that \eqref{H_{a+1} goal} holds, completing the proof of assertion (1) in the theorem.

Now we assume that $k+1\leq a\leq d_{1}+d_{2}-2$. According to \autoref{colon generators when equigenerated}, we have
\begin{equation*}
    (x^{d_{2}},y^{d_{2}}):(x+y)^{a-k}=\begin{cases}
        (F_{1,d_{2},a-k,0},\, F_{2,d_{2},a-k,0}) & \text{when}\ a-k\ \text{is odd},\\
        (G_{1,d_{2},a-k,0},\, G_{2,d_{2},a-k,0}) & \text{when}\ a-k\ \text{is even}.
    \end{cases}
\end{equation*}
Suppose $a-k\geq 2$ is even. Then $G_{1,d_{2},a-k,k},\, G_{2,d_{2},a-k,k}\in(x^{d_{1}},y^{d_{2}}):(x+y)^{a}$ by \autoref{partial derivs of generators} and \autoref{partial deriv belongs to colon ideal}. Furthermore,
\small
\begin{equation*}
    \deg(G_{1,d_{2},a-k,k})=d_{2}-\frac{a-k}{2}-k=d_{2}-\frac{a+k}{2},\ \ \text{and similarly}\ \ \deg(G_{2,d_{2},a-k,k})=d_{2}-\frac{a+k}{2}.
\end{equation*}
\normalsize
\autoref{in colon ideal with right degrees implies generation} now guarantees that $G_{1,d_{2},a-k,k}$ and $G_{2,d_{2},a-k,k}$ generate the colon ideal. The argument proceeds identically (with the obvious necessary modifications) to show that $F_{1,d_{2},a-k,k}$ and $F_{2,d_{2},a-k,k}$ generate the colon ideal when $a-k\geq 3$ is odd.

It remains to prove the theorem when $a-k=1$ (i.e. $a=k+1$) and $k\geq 1$. The content of \autoref{partial derivs of generators} and \autoref{partial deriv belongs to colon ideal} still applies ``meaningfully'' to $F_{1,d_{2},1,k}$; more precisely, we have $F_{1,d_{2},1,k}$ is a nonzero element of $(x^{d_{2}-k},y^{d_{2}}):(x+y)^{1+k}$. For $F_{2,d_{2},1,k}$, we first verify directly that
\begin{equation*}
    F_{2,d_{2},1,1}\in(x^{d_{2}-1},y^{d_{2}}):(x+y)^{2}.
\end{equation*}
Toward this end, we compute the coefficient of $x^{d_{2}-i}y^{i+1}$ in $(x+y)^{2}F_{2,d_{2},1,1}$. Being a multiple of $y^{2}$, we look at the $(i-1)^{\text{st}}$, $(i-2)^{\text{nd}}$, and $(i-3)^{\text{rd}}$ terms in the expansion:
\begin{align*}
    & \text{cff}_{(x+y)^{2}F_{2,d_{2},1,1}}(x^{d_{2}-i}y^{i+1})=(-1)^{i-1}i+2(-1)^{i-2}(i-1)+(-1)^{i-3}(i-2) \\
    \implies\quad & \text{cff}_{(x+y)^{2}F_{2,d_{2},1,1}}(x^{d_{2}-i}y^{i+1})=0\ \text{for all}\ i\in\{2,\ldots,d_{2}-2\}.
\end{align*}
Therefore,
\small
\begin{equation*}
    (x+y)^{2}F_{2,d_{2},1,1}=x^{d_{2}-1}y^{2}+\underbrace{2(-1)^{d_{2}-3}(d_{2}-2)xy^{d_{2}}+(-1)^{d_{2}-4}(d_{2}-3)xy^{d_{2}}+(-1)^{d_{2}-3}(d_{2}-2)y^{d_{2}+1}}_{\text{multiple of}\ y^{d_{2}}}.
\end{equation*}
\normalsize
We obtain $F_{2,d_{2},1,1}\in(x^{d_{2}-1},y^{d_{2}}):(x+y)^{2}$ as desired. Now \autoref{partial derivs of generators} and (the proof of) \autoref{partial deriv belongs to colon ideal} imply $F_{2,d_{2},1,k}\in(x^{d_{2}-k},y^{d_{2}}):(x+y)^{1+k}$. From here, we compute:
\begin{equation*}
    \deg(F_{1,d_{2},1,k})=d_{2}-1-k=d_{2}-a,\quad \text{and similarly}\quad \deg(F_{2,d_{2},1,k})=d_{2}-a+1.
\end{equation*}
Thus, we may use \autoref{in colon ideal with right degrees implies generation} to conclude $(x^{d_{2}-k},y^{d_{2}}):(x+y)^{1+k}=(F_{1,d_{2},1,k},F_{2,d_{2},1,k})$.
\end{proof}

\section{Application to WLP in 3-Variable Monomial ACIs}

In this section we consider the question of classifying the values of $(a_1, a_2, a_3, t)$ for which WLP holds for the level almost complete intersection ring $A=\F[x,y,z]/I$, where the ideal $I=(x^{t+a_{1}},y^{t+a_{2}},z^{t+a_{3}},x^{a_{1}}y^{a_{2}}z^{a_{3}})$. We assume this generating set is minimal and that $a_{1},a_{2},a_{3}>0$. Because WLP is invariant under permutation of variables, we may without loss of generality relabel the variables so that $a_{1}\leq a_{2}\leq a_{3}$. In this ordering, the condition $a_3<2(a_1+a_2)$ below says (in the original unordered notation) that the largest of the three exponents is less than twice the sum of the other two. As an application of the results from the previous section, our main result in this section is the following:

\begin{thrm}\label{polynomials}
Let $d_{i}=t+a_{i}$ for $i\in\{1,2,3\}$. For fixed values of $a_1\leq a_2\leq a_3$ with $a_1+a_2+a_3$ divisible by $3$ and $a_3<2(a_1+a_2)$, there exist polynomials $P_1(t)$ and $P_2(t)$ such that:

{\rm a.} For $t_{0}\ge \frac{a_1+a_2+a_3}{3}$ even, WLP fails if and only if $P_1(t_{0})=0$.

{\rm b.} For $t_{0}\ge \frac{a_1+a_2+a_3}{3}$ odd, WLP fails if and only if $P_2(t_{0})=0$.

In particular, for fixed $a_1, a_2, a_3$ and fixed parity for $t_{0}$, either WLP fails for all $t_{0}\geq\frac{a_{1}+a_{2}+a_{3}}{3}$ of that parity (when $P_{i}$ is identically zero) or fails for at most $\deg(P_{i})$ such $t_{0}$.
\end{thrm}

We can write the polynomials in this theorem explicitly for small values of $a_1, a_2, a_3$ and prove that they do not have any ``large'' roots, meaning that WLP holds for all values of $t\geq\frac{a_{1}+a_{2}+a_{3}}{3}$ (see \autoref{example}). While the proof of this theorem will be written after relabeling so that $a_{1}\leq a_{2}\leq a_{3}$, note that the conclusion applies to any unordered triple of positive exponents satisfying the corresponding condition on the largest exponent relative to the other two.

In anticipation of the theorem's proof, we have two preparatory lemmas:

\begin{lem}\label{lem1}
Let $I=(x^{d_1}, y^{d_2}, z^{d_3}, x^{a_1}y^{a_2}z^{a_3})\subset \F[x, y, z]$, and $L=x+y+z$.
For a polynomial $f\in \F[x, y, z]$, we denote the image of $f$ under the homomorphism that sends $z$ to $-(x+y)$ by $\tilde{f}\in \F[x, y]$. Fix any degree $d$.
\begin{enumerate}[label=(\alph*)]
    \item Suppose $d$ is such that
    \begin{equation*}
        \times L:\big(\F[x,y,z]/(x^{d_{1}},y^{d_{2}},z^{d_{3}})\big)_{d-1}\ \longrightarrow\ \big(\F[x,y,z]/(x^{d_{1}},y^{d_{2}},z^{d_{3}})\big)_{d}\quad \text{is injective}.
    \end{equation*}
    If we have a homogeneous relation of degree $d$ (meaning each of the terms is homogeneous of degree $d$)
    \begin{equation}\label{wlp relation}
        FL=Ax^{d_1}+By^{d_2}+Cz^{d_3}+Dx^{a_1}y^{a_2}z^{a_3}
    \end{equation}
    with $F, A, B, C, D\in \F[x,y,z]$ and $F\notin I$, then $\tilde{D}\notin (x^{d_1-a_1}, y^{d_2-a_2} , (x+y)^{d_3-a_3})$.
    \item Conversely, if there exist homogeneous forms $A', B', C', D'\in \F[x, y]$ such that 
\begin{equation*}
    A'x^{d_1}+B'y^{d_2} + C'(x+y)^{d_3}+D'x^{a_1}y^{a_2}(x+y)^{a_3}=0
\end{equation*}
is a homogeneous relation of degree $d$ and 
$D'\notin (x^{d_1-a_1}, y^{d_2-a_2}, (x+y)^{d_3-a_3})$, then there exists a homogeneous relation of degree $d$ as in \eqref{wlp relation} with $F\notin I$.
\end{enumerate}
\end{lem}

\begin{proof}
For (a), we argue by contrapositive. We have
\begin{equation*}
    \tilde{D}\in (x^{d_1-a_1}, y^{d_2-a_2} , (x+y)^{d_3-a_3})\quad \iff\quad D\in (x^{d_1-a_1}, y^{d_2-a_2}, z^{d_3-a_3}, L).
\end{equation*}
Say a homogeneous form $G$ is such that $D-GL\in (x^{d_1-a_1}, y^{d_2-a_2}, z^{d_3-a_3})$. Then 
\begin{equation*}
    (F-Gx^{a_1}y^{a_2}z^{a_3})L\in (x^{d_1}, y^{d_2}, z^{d_3}).
\end{equation*}
The assumption that $\times L$ is injective implies $F-Gx^{a_1}y^{a_2}z^{a_3} \in (x^{d_1}, y^{d_2}, z^{d_3})$, whence we obtain $F\in I$.

Part (b) proceeds by contrapositive as well. Let
\begin{equation*}
    F:=C'\cdot \frac{(x+y)^{d_3} - (-z)^{d_3}}{x+y+z} + D'x^{a_{1}}y^{a_{2}} \cdot \frac{(x+y)^{a_3}-(-z)^{a_3}}{x+y+z}.
\end{equation*}
It is clear that $FL\in I$ and $\deg(F)=d-1$. Further, we can rewrite
\begin{align*}
    F = &\ C'\big((x+y)^{d_3-1}+(-z)(x+y)^{d_3-2} + \cdots + (-z)^{d_3-1}\big)\, + \\
    &\  D'x^{a_1}y^{a_2}\big((x+y)^{a_3-1}+(-z)(x+y)^{a_3-2} + \cdots + (-z)^{a_3-1}\big).
\end{align*}
Assume $F\in I$. Looking at the terms of $F$ containing $z^{d_3-1}$, we see that $F\in I$ implies $C'\in (x^{d_1}, y^{d_2}, x^{a_1}y^{a_2})$. Next, look at the terms of $F$ containing $z^{a_3-1}$; we see that $F\in I$ implies $C'(x+y)^{d_3-a_3} + D'x^{a_1}y^{a_2}\in (x^{d_1}, y^{d_2})$. Thus, we obtain
\begin{equation*}
    D'x^{a_1}y^{a_2}\in (x^{d_1}, y^{d_2})+(x+y)^{d_3-a_3}(x^{d_1}, y^{d_2}, x^{a_1}y^{a_2})\quad \text{and note that}
\end{equation*}
\begin{equation*}
    (x^{d_1}, y^{d_2})+(x+y)^{d_3-a_3}(x^{d_1}, y^{d_2}, x^{a_1}y^{a_2})=(x^{d_1}, y^{d_2}, (x+y)^{d_3-a_3}x^{a_1}y^{a_2}).
\end{equation*}
This implies that there exists $H\in \F[x, y]$ such that
\begin{equation*}
    (D'-(x+y)^{d_3-a_3}H)x^{a_1}y^{a_2}\in (x^{d_1}, y^{d_2})
\end{equation*}
and thence $D'\in (x^{d_1-a_1}, y^{d_2-a_2}, (x+y)^{d_3-a_3})$.
\end{proof}

\begin{lem}\label{lem2}
Let $d_i=t+a_i$ for $1\le i \le 3$ and $d=t+\frac{2(a_1+a_2+a_3)}{3}-1$. Assume that $3$ divides $a_1+a_2+a_3$ and $a_3\leq 2(a_1+a_2)$.

Let $I=(x^{d_1}, y^{d_2}, z^{d_3}, x^{a_1}y^{a_2}z^{a_3})$. 

a. Assume that $FL=Ax^{d_1}+By^{d_2}+Cz^{d_3}+Dx^{a_1}y^{a_2}z^{a_3}$ is a homogeneous relation of degree $d$ with $F\notin I$. Then $\tilde{C}\ne 0$.

b. Assume there is a homogeneous relation of degree $d$ $A'x^{d_1}+B'y^{d_2}+C'(x+y)^{d_3}+D'x^{a_1}y^{a_2}(x+y)^{a_3}=0$  with $A', B', C', D'\in \F[x, y]$ and $C'\ne 0$. Then there is a homogeneous relation of degree $d$ as in part a. with $F\notin I$.
\end{lem}

\begin{proof}

a. Assume by way of contradiction that $\tilde{C}=0$. This means that $C$ is a multiple of $L$, say $C=KL$. Rewrite equation \eqref{wlp relation} as 
$$
(F-Kz^{d_3})L= Ax^{d_1}+By^{d_2}+Dx^{a_1}y^{a_2}z^{a_3}$$
which implies $\tilde{D}x^{a_1}y^{a_2}(x+y)^{a_3}\in (x^{d_1}, y^{d_2})$, i.e. $\tilde{D}\in (x^{d_1-a_1}, y^{d_2-a_2}):(x+y)^{a_3}=(x^t, y^t):(x+y)^{a_3}$.

On the other hand, we have $\mathrm{deg}(\tilde{D})=t-\frac{a_1+a_2+a_3+3}{3}$, and the smallest degree of a generator of $(x^t, y^t):(x+y)^{a_3}$ is $t-\left\lceil \frac{a_3}{2}\right\rceil $. The assumption that $2(a_1+a_2)\geq a_3$ implies that $t-\left\lceil \frac{a_3}{2}\right\rceil > t-\frac{a_1+a_2+a_3+3}{3}$, and therefore $\tilde{D}=0$. \autoref{lem1} now implies $F\in I$, which is a contradiction.

\bigskip

b. Note first that if $a_{3}=2(a_{1}+a_{2})$ then we have $\deg(C')=-1$, so no such relation with $C'\neq 0$ exists and the conclusion is vacuously true. We therefore assume $a_{3}<2(a_{1}+a_{2})$. Define
$$
F:=C'\cdot \frac{(x+y)^{d_3} - (-z)^{d_3}}{x+y+z} + D'x^{a_{1}}y^{a_{2}}\cdot \frac{(x+y)^{a_3}-(-z)^{a_3}}{x+y+z}
$$ as in the proof of \autoref{lem1}.

We saw in the proof of \autoref{lem1} that $F\in I$ implies $C'\in (x^{d_1}, y^{d_2}, x^{a_1}y^{a_2})$. 
However, $\mathrm{deg}(C')=d-d_3=\frac{2(a_1+a_2)-a_3}{3}-1
<\mathrm{min}\{d_1, d_2, a_1+a_2\}$, so we must have $C'=0$.    
\end{proof}

\begin{cor}\label{a=0}
Retain the setup assumptions from \autoref{lem2}. If $a_3=2(a_1+a_2)$, then WLP holds for all $t$.
\end{cor}
\begin{proof}
Part (b)(1) of Theorem 4.10 in \cite{Cook-Nagel-21} gives the only degrees in which $\times L$ can fail to have maximal rank, and Lemma 7.1 in \cite{Mig-Mir-Nag-11} says the Hilbert function has the same value in both of these degrees. Therefore, WLP fails if and only if there exists a homogeneous relation
\begin{equation*}
FL=Ax^{t+a_{1}}+By^{t+a_{2}}+Cz^{t+a_{3}}+Dx^{a_1}y^{a_2}z^{a_3}
\end{equation*}
of degree $d:=t+\frac{2(a_1+a_2+a_3)}{3}-1$ with $F\notin I$.
Let $a:=\frac{2(a_1+a_2)-a_3}{3}$, so $\mathrm{deg}(C)=a-1$.
In this case $a=0$ so $C$ must be 0. This is a contradiction due to part a. of \autoref{lem2}.
\end{proof}

We are now ready to prove \autoref{polynomials}.

\begin{proof} 

As mentioned in the proof of the preceding corollary, we know that WLP fails if and only if there exists a homogeneous relation 
\begin{equation*}
FL=Ax^{t+a_{1}}+By^{t+a_{2}}+Cz^{t+a_{3}}+Dx^{a_1}y^{a_2}z^{a_3}
\end{equation*}
of degree $d:=t+\frac{2(a_1+a_2+a_3)}{3}-1$ with $F\notin I$.
Let $a:=\frac{2(a_1+a_2)-a_3}{3}$, so $\mathrm{deg}(C)=a-1$.

Assuming WLP fails,
apply the homomorphism that sends $z$ to $-(x+y)$ to \eqref{wlp relation} to obtain
\begin{equation}\label{tild}
\tilde{A}x^{d_1}+\tilde{B}y^{d_2}+\tilde{C}(x+y)^{d_3}+\tilde{D}x^{a_1}y^{a_2}(x+y)^{a_3}=0,
\end{equation}
and therefore
$$
\tilde{C}(x+y)^t \in (x^{d_1}, y^{d_2}):(x+y)^{a_3} + (x^{a_1}y^{a_2}),
$$
Let $F_1, F_2$ denote the generators of $(x^{d_1}, y^{d_2}):(x+y)^{a_3}$. We can write 
\begin{equation}\label{sys}
\tilde{C}(x+y)^t = H_1F_1 + H_2F_2 \ \ \ \mathrm{mod}\, (x^{a_1}y^{a_2})
\end{equation}

We know that 
$$\displaystyle \mathrm{deg}(F_1)= \left\lfloor \frac{d_1+d_2-a_3}{2}\right\rfloor, \ \ \ \displaystyle \mathrm{deg}(F_2)=\left\lceil \frac{d_1+d_2-a_3}{2}\right\rceil ,$$ and therefore
$$
\mathrm{deg}(H_1)=t+a-1-\left\lfloor \frac{d_1+d_2-a_3}{2}\right\rfloor =a-1-\left\lfloor\frac{a_1+a_2-a_3}{2} \right\rfloor = \left\lceil \frac{a_1+a_2+a_3}{6} \right\rceil -1
$$
and 
$$
\mathrm{deg}(H_2)=\left\lfloor \frac{a_1+a_2+a_3}{6}\right\rfloor-1.
$$
We view the coefficients of $\tilde{C}, H_1, H_2$ as unknowns. There are a total of $a+ \frac{a_1+a_2+a_3}{3}=a_1+a_2$ unknowns. 
Equation (\ref{sys}) translates into a system of linear equations in these unknowns: for each monomial of degree $a+t-1$ not divisible by $x^{a_1}y^{a_2}$ we get an equation by setting the coefficient of that monomial on the left hand side equal to the coefficient of that monomial on the right hand side. More precisely, we get one equation for each of the following monomials: $x^{t+a-1}, x^{t+a-2}y, \ldots, x^{ t+a-a_2}y^{a_2-1}, x^{a_1-1}y^{t+a-a_1}, \ldots, xy^{t+a-2}, y^{t+a-1}$.

The assumption $\displaystyle t>\frac{a_1+a_2+a_3}{3}-1$ is equivalent to  $t+a-a_2+1>a_1-1$,  so that the monomials listed above are all distinct, and we have a total of $a_1+a_2$ equations. Therefore we have a linear system of $a_1+a_2$ equations with $a_1+a_2$ unknowns. 

We claim that the system given by (\ref{sys}) has a nontrivial solution if and only if there exists $\tilde{C}\ne 0$ that satisfies $(\ref{sys})$. Indeed, if there is a nontrivial solution with $\tilde{C}=0$ then we have $H_1F_1+H_2F_2\in (x^{a_1}y^{a_2})$. This would mean that there exists a nonzero polynomial $K$ of degree $t+a-1-(a_1+a_2)=t-\frac{a_1+a_2+a_3+3}{3}$
such that $Kx^{a_1}y^{a_2}\in (F_1, F_2)=(x^{d_1}, y^{d_2}):(x+y)^{a_3}$, and therefore $K\in (x^t, y^t):(x+y)^{a_3}$.
However, the smallest degree of a generator of $(x^t, y^t):(x+y)^{a_3}$ is $t-\lceil \frac{a_3}{2}\rceil $. The assumption that $2(a_1+a_2)>a_3$ implies that $t-\lceil \frac{a_3}{2}\rceil > t-\frac{a_1+a_2+a_3+3}{3}$. This is a contradiction.

According to \autoref{lem2}, we now have that failure of WLP is equivalent to existence of a nontrivial solution for the linear system of equations given by (\ref{sys}), which is equivalent to the vanishing of the determinant of an $(a_1+a_2)\times (a_1+a_2)$ matrix. Note that the entries in this matrix come from the relevant coefficients in $F_1$ and $F_2$, and are polynomials in $t$. We get two different polynomials according to the parity of $t$ because the coefficients of $F_1$ and $F_2$ involve $(-1)^t$.
\end{proof}

\begin{exam}\label{example}
\normalfont We offer a concrete illustration of \autoref{polynomials}. Let $(a_{1},a_{2},a_{3})=(3,7,14)$. When $t=9$, the comment after Question 7.12 in \cite{Mig-Mir-Nag-11} (or calculating the annihilator of $L=x+y+z$ from degree $23$ to $24$ in \texttt{Macaulay2}) tells us the associated ideal fails WLP; we will see this failure manifest in the determinant. We have $a=2$ here and so note the following degrees: $\deg(C)=1$, $\deg(H_{1})=3$, and $\deg(H_{2})=3$. Further, $k=4$ in this setting, whence $k+1\leq a_{3}$ and $a_{3}-k=10$ is even. By \autoref{thrm:colon generators in general}, the generators of $(x^{t+3},y^{t+7}):(x+y)^{14}$ are
\begin{align*}
    G_{1,t+7,10,4} & =y\sum_{i=0}^{t+1}(-1)^{i}\binom{t+5-i}{8}\binom{5+i}{5}x^{t-3-i}y^{i}\qquad\text{and} \\
    G_{2,t+7,10,4} & =\sum_{i=0}^{t+2}(-1)^{i}\binom{t+6-i}{8}\binom{4+i}{4}(i-1)x^{t-2-i}y^{i}.
\end{align*}
Our relation in this context reads as $H_{1}G_{1,t+7,10,4}+H_{2}G_{2,t+7,10,4}=\widetilde{C}(x+y)^{t}$. Modding by $x^{3}y^{7}$ and taking the relevant degrees into consideration, we obtain a $10\times 10$ matrix. For ease of readability, we write this matrix in three parts: $A_{1}$ encodes the coefficients from $H_{1}G_{1,t+7,10,4}$, $A_{2}$ the coefficients from $H_{2}G_{2,t+7,10,4}$, and $A_{3}$ the coefficients from $\widetilde{C}(x+y)^{t}$.
\begin{equation*}
    A_{1}=\begin{bmatrix}[1.5]
        0 & 0 & 0 & 0 \\
        \binom{t+5}{8} & 0 & 0 & 0 \\
        -\binom{t+4}{8}\binom{6}{5} & \binom{t+5}{8} & 0 & 0 \\
        \binom{t+3}{8}\binom{7}{5} & -\binom{t+4}{8}\binom{6}{5} & \binom{t+5}{8} & 0 \\
        -\binom{t+2}{8}\binom{8}{5} & \binom{t+3}{8}\binom{7}{5} & -\binom{t+4}{8}\binom{6}{5} & \binom{t+5}{8} \\
        \binom{t+1}{8}\binom{9}{5} & -\binom{t+2}{8}\binom{8}{5} & \binom{t+3}{8}\binom{7}{5} & -\binom{t+4}{8}\binom{6}{5} \\
        -\binom{t}{8}\binom{10}{5} & \binom{t+1}{8}\binom{9}{5} & -\binom{t+2}{8}\binom{8}{5} & \binom{t+3}{8}\binom{7}{5} \\
        0 & (-1)^{t-3}\binom{t+2}{5} & (-1)^{t-4}\binom{9}{8}\binom{t+1}{5} & (-1)^{t-5}\binom{10}{8}\binom{t}{5} \\
        0 & 0 & (-1)^{t-3}\binom{t+2}{5} & (-1)^{t-4}\binom{9}{8}\binom{t+1}{5} \\
        0 & 0 & 0 & (-1)^{t-3}\binom{t+2}{5}
    \end{bmatrix}
\end{equation*}
\begin{equation*}
    A_{2}=\begin{bmatrix}[1.5]
        -\binom{t+6}{8} & 0 & 0 & 0 \\
        0 & -\binom{t+6}{8} & 0 & 0 \\
        \binom{t+4}{8}\binom{6}{4} & 0 & -\binom{t+6}{8} & 0 \\
        -2\binom{t+3}{8}\binom{7}{4} & \binom{t+4}{8}\binom{6}{4} & 0 & -\binom{t+6}{8} \\
        3\binom{t+2}{8}\binom{8}{4} & -2\binom{t+3}{8}\binom{7}{4} & \binom{t+4}{8}\binom{6}{4} & 0 \\
        -4\binom{t+1}{8}\binom{9}{4} & 3\binom{t+2}{8}\binom{8}{4} & -2\binom{t+3}{8}\binom{7}{4} & \binom{t+4}{8}\binom{6}{4} \\
        5\binom{t}{8}\binom{10}{4} & -4\binom{t+1}{8}\binom{9}{4} & 3\binom{t+2}{8}\binom{8}{4} & -2\binom{t+3}{8}\binom{7}{4} \\
        0 & (-1)^{t-2}(t-3)\binom{t+2}{4} & (-1)^{t-3}(t-4)\binom{9}{8}\binom{t+1}{4} & (-1)^{t-4}(t-5)\binom{10}{8}\binom{t}{4} \\
        0 & 0 & (-1)^{t-2}(t-3)\binom{t+2}{4} & (-1)^{t-3}(t-4)\binom{9}{8}\binom{t+1}{4} \\
        0 & 0 & 0 & (-1)^{t-2}(t-3)\binom{t+2}{4}
    \end{bmatrix}
\end{equation*}
\vspace{5pt}
\begin{equation*}
    A_{3}=\begin{bmatrix}[1.5]
        1 & 0 \\
        \binom{t}{1} & 1 \\
        \binom{t}{2} & \binom{t}{1} \\
        \binom{t}{3} & \binom{t}{2} \\
        \binom{t}{4} & \binom{t}{3} \\
        \binom{t}{5} & \binom{t}{4} \\
        \binom{t}{6} & \binom{t}{5} \\
        \binom{t}{t-1} & \binom{t}{t-2} \\
        1 & \binom{t}{t-1} \\
        0 & 1
    \end{bmatrix}
\end{equation*}
Let $A$ be the matrix whose $10$ columns are formed from $A_{1}$, $A_{2}$, and $A_{3}$. Using \texttt{Macaulay2}, we find (up to multiplication by a scalar) its determinant when $t$ is odd:
\small
\begin{align*}
    \det(A)= &\ t^{8}(t-9)(t-7)(t-5)^{3}(t-4)^{2}(t-3)^{5}(t-2)^{4}(t-1)^{10}(t+1)^{7}(t+2)^{4}(t+3)^{2}(t+4) \\
    &\, (t+8)^{2}(t+9)(t+10)^{4}(t+11)^{3}(t+12)^{2}(t+14)^{2}(7t^{3}+115t^{2}+504t+36).
\end{align*}
\normalsize
Note the presence of $t=9$ as a root of this polynomial! (The other positive roots are all less than $\frac{a_{1}+a_{2}+a_{3}}{3}=8$ and so fall outside our theorem's purview.) Moreover, no positive integer root is larger than $9$, which means WLP holds for $I=(x^{3+t},y^{7+t},z^{14+t},x^{3}y^{7}z^{14})$ for all odd values of $t\geq 11$.
\end{exam}

In a different direction, we now focus on proving \autoref{cook-nagel} in some cases that are close to borderline in the sense that the values of $a_1, a_2, a_3, t$ are such that the inequalities in the assumptions of \autoref{cook-nagel} are satisfied, but are close to equalities.
These assumptions are: $0<a_1\le a_2 \le a_3 \le 2(a_1+a_2)$ and $t\ge \frac{a_1+a_2+a_3}{3}$. We note that the conjecture is known for $t=\frac{a_1+a_2+a_3}{3}$; since $t$ and $a_1+a_2+a_3$ have the same parity in this instance, WLP holds per part (i) of Proposition 4.14 in \cite{Cook-Nagel-21}. Hence, the smallest value of $t$ where the conjecture is not known is $t=\frac{a_1+a_2+a_3}{3}+1$. According to \autoref{polynomials}, proving that WLP holds for this value of $t$ will imply that the conjecture can fail for at most finitely many values of $t$ when the values of $a_1, a_2, a_3$ are fixed. 

Due to the assumption that $a_1+a_2+a_3$ is a multiple of 3, we can write $a_3=2(a_1+a_2)-3a$ where $a$ is a nonnegative integer. The next result shows that \autoref{cook-nagel} holds for the value $t=\frac{a_1+a_2+a_3}{3}+1$ when $a\le 3$.

\begin{thrm}\label{borderline}
Consider the ideal $I=(x^{a_1+t}, y^{a_2+t}, z^{a_3+t}, x^{a_1}y^{a_2}z^{a_3})\subseteq\F[x, y, z]$ where $a_1, a_2, a_3, t$ satisfy the three conditions stated in equation \eqref{conditions}, and write $a_3=2(a_1+a_2)-3a$ where $a$ is a nonnegative integer. Assume $t=\frac{a_1+a_2+a_3}{3}+1$ and $a\leq 3$. If $(a_1, a_2, a_3, t)$ is not $(2, 9, 13, 9)$ or $(3, 7, 14, 9)$, then $\F[x, y, z]/I$ has WLP except when $t$ is even, $a_1+a_2+a_3$ is odd, and $a_{1}$, $a_{2}$, $a_{3}$ are not all distinct.
\end{thrm}

\begin{proof}
The case $a=0$ is known from \autoref{a=0}.
Assume $a\ge 1$.

Recall that the pairing of Theorem 4.10 in \cite{Cook-Nagel-21} with Lemma 7.1 in \cite{Mig-Mir-Nag-11} implies that WLP fails if and only if there is a homogeneous relation of degree $d$
\begin{equation}\label{en}
FL=Ax^{t+a_1}+By^{t+a_2}+Cz^{t+a_3}+Dx^{a_1}y^{a_2}z^{a_3}
\end{equation}
with $F\notin I$, where $d=t+\frac{2(a_1+a_2+a_3)}{3}-1=a_1+a_2+a_3$.
We have $\mathrm{deg}(C)=d-(t+a_3)=a-1$ and $\mathrm{deg}(D)=d-(a_1+a_2+a_3)=0$, so $D$ is a constant. We note that $D\ne 0$, since otherwise we would have 
$FL\in (x^{t+a_1}, y^{t+a_2}, z^{t+a_3})$; since $d$ is exactly half of the socle degree of the complete intersection $(x^{t+a_1}, y^{t+a_2}, z^{t+a_3})$ and $\deg(F)=d-1$, this would contradict the WLP for monomial complete intersections. Therefore, we may take $D=1$.

Note that the assumption implies $a\le a_1$, for otherwise $a_3=2a_1+2a_2-3a<2a_1-a_2\le a_2$, a contradiction.

Define $\Delta$ to be the operator $\frac{\partial}{\partial x} -\frac{\partial}{\partial y}$ (applied to a polynomial) and $\Delta ^k$ means that this operator is applied repeatedly $k$ times. 

Furthermore, we define an operator $\tau ^k$ recursively as follows:
$
\tau^0 D:=D,
$
\small
$$ \tau ^k D:= \left \lbrace \begin{array}{lcc} xy \Delta (\tau ^{k-1} D)+((a_1-k+1)y-(a_2-k+1)x)(\tau ^{k-1} D) & \mathrm{if} & 1\le k \le a_1+1, \\ 
y\Delta(\tau^{k-1}D)-(a_2-k+1)\tau^{k-1} D & \mathrm{if} & a_1+1<k\le a_2+1.\\ \end{array} \right.
$$
\normalsize
Note that $\mathrm{deg}(\tau^kD)$ is $k+\mathrm{deg}(D)$ if $1\le k \le \mathrm{min}\{a_1+1, a_2+1\}$, and it is $a_1+\mathrm{deg}(D)$ if $a_1+1< k\le a_2+1$.
The definition of $\tau^kD$ is so that
$$
\Delta^k(x^{a_1}y^{a_2}z^{a_3}D)=\left\lbrace \begin{array}{lcc} x^{a_1-k}y^{a_2-k}z^{a_3}(\tau^kD) & \mathrm{if} & k\le a_1, \\ y^{a_2-k}z^{a_3}(\tau^kD) & \mathrm{if} & a_1<k \le a_2.\\ \end{array}\right.$$

Applying $\Delta$ in equation (\ref{en}) yields 
$$
(\Delta ^k F) L = A' x^{t+a_1-k} + B'y^{t+a_2-k} + (\Delta ^k C)z^{t+a_3} +(\tau^kD)x^{a_1-k}y^{a_2-k}z^{a_3}
$$
for all $1\le k\le a_2$, where $A', B'\in \F[x, y, z]$ and we use the convention that $x^{a_1-k}=1$ if $a_1<k$.

Since $\mathrm{deg}(C)=a-1$, we have $\Delta^aC=0$, and we get
\begin{equation}\label{tilde}
(\Delta ^a F) L = A' x^{t+a_1-a} + B'y^{t+a_2-a} +(\tau^aD)x^{a_1-a}y^{a_2-a}z^{a_3}.
\end{equation}

Apply the homomorphism that sends $z$ to $-(x+y)$. Recall that $\tilde{f}$ denotes the image of a polynomial $f\in \F[x, y, z]$ under this homomorphism.
Since $D=1$, $\tau^kD$ is a polynomial in $x, y$ and therefore it is unaffected by this homomorphism. Note that the left hand side of (\ref{tilde}) becomes zero. We have 
$$
(\tau ^aD) x^{a_1-a}y^{a_2-a}(x+y)^{a_3}\in (x^{t+a_1-a}, y^{t+a_2-a}) \ \mathrm{if} \ a_1\ge a,
$$
and 
$$
(\tau ^aD) y^{a_2-a}(x+y)^{a_3}\in (x^{t+a_1-a}, y^{t+a_2-a}) \ \mathrm{if} \ a_1<a.
$$
Equivalently,  
\begin{equation}\label{cond1}
(\tau ^a D) (x+y)^{a_3} \in (x^t, y^t) \ \mathrm{if} \ a_1\ge a, \mathrm{and} 
\end{equation}
\begin{equation}\label{cond2}
(\tau ^a D) (x+y)^{a_3} \in (x^{t+a_1-a}, y^t) \ \mathrm{if} \ a_1<a
\end{equation}
If $a\le a_1$, we have $\mathrm{deg}((\tau ^a D) (x+y)^{a_3})=a+a_3=2t-2$, so (\ref{cond1}) is equivalent to the coefficient of $x^{t-1}y^{t-1}$ being equal to zero. 

If $a_1<a$, the degree of $(\tau ^a D) (x+y)^{a_3}$ is $a_1+a_3=2t+a_1-a-2$, so (\ref{cond2}) is equivalent to the coefficient of $x^{t+a_1-a-1}y^{t-1}$ being equal to zero.

Our strategy is to write the relevant coefficient as a (multiple of a) polynomial in the variables $a_1, a_2$, which we will denote $\mathcal F(a_1, a_2)$, and analyze the integer solutions for $\mathcal F(a_1, a_2)=0.$

$$
\tau^1 D=a_1y-a_2x,
$$
$$\tau^2 D=\left\lbrace \begin{array}{lcc} (a_2^2-a_2)x^2-2a_1a_2xy+(a_1^2-a_1)y^2 & \mathrm{if} & a_1\ge 2\\  (a_2^2-a_2)x-(a_1a_2+a_2)y & \mathrm{if} & a_1=1 \end{array} \right.$$
\footnotesize
\begin{equation*}
    \tau^{3}D=\begin{cases}
        (-a_{2}^{3}+3a_{2}^{2}-2a_{2})x^{3}+(3a_{1}a_{2}^{2}-3a_{1}a_{2})x^{2}y+(-3a_{1}^{2}a_{2}+3a_{1}a_{2})xy^{2}+(a_{1}^{3}-3a_{1}^{2}+2a_{1})y^{3} & \text{if}\ a_{1}\geq 3, \\
        (-a_{2}^{3}+3a_2^{2}-2a_{2})x^{2}+(2a_{1}a_{2}^{2}-2a_{1}a_{2}+2a_{2}^{2}-2a_{2})xy-(a_{1}^{2}a_{2}+a_{1}a_{2})y^{2} & \text{if}\ a_{1}=2, \\
        (-a_{2}^{3}+3a_{2}^{2}-2a_{2})x+(a_{1}a_{2}^{2}-a_{1}a_{2}+2a_{2}^{2}-2a_{2})y & \text{if}\ a_{1}=1.
    \end{cases}
\end{equation*}
\normalsize
We will use the notation $\mathrm{cff}(x^iy^j, \tau^aD)$ to indicate the coefficient of $x^iy^j$ in $\tau^aD$.

{\bf The case $a=1$:} the coefficient of $x^{t-1}y^{t-1}$ in $(\tau ^1 D)(x+y)^{a_3}$ is equal to 
$$
\binom{a_3}{\frac{a_3-1}{2}}\mathrm{cff}(x, \tau^1D)+\binom{a_3}{\frac{a_3+1}{2}}\mathrm{cff}(y, \tau^1D)=A(a_2-a_1)
$$
where $A=\binom{a_3}{\frac{a_3+1}{2}}=\binom{a_3}{\frac{a_3-1}{2}}$.
We have $\mathcal{F}(a_1, a_2):=a_1-a_2=0\Leftrightarrow a_1=a_2$.

{\bf The case $a=2$:}
We need to consider the cases $a_1\ge 2$ and $a_1=1$ separately.

Assume $a_1\ge 2$. The coefficient of $x^{t-1}y^{t-1}$ in $(\tau^2 D)(x+y)^{a_3}$ is equal to 
$$
\mathcal C:=\binom{a_3}{\frac{a_3}{2}-1}\mathrm{cff}(x^2, \tau^2D)+\binom{a_3}{\frac{a_3}{2}} \mathrm{cff}(xy, \tau^2D)+\binom{a_3}{\frac{a_3}{2}+1}\mathrm{cff}(y^2, \tau^2D)
$$

Let $A:=\binom{a_3}{\frac{a_3}{2}}$. Then $\binom{a_3}{\frac{a_3}{2}-1}=\binom{a_3}{\frac{a_3}{2}+1}= \frac{(\frac{a_3}{2})}{(\frac{a_3}{2}+1)} A=\frac{a_3}{a_3+2}A$, therefore 
$$
\mathcal C=\frac{A}{a_3+2}(a_3(\mathrm{cff}(x^2, \tau^2D)+\mathrm{cff}(y^2, \tau^2D))+(a_3+2)\mathrm{cff}(xy, \tau^2D))
$$
Substituting $a_3=2(a_1+a_2)-6$ and dividing by 2, we have 
$$\mathcal C=0 \Leftrightarrow \mathcal{F}(a_1, a_2):=(a_1+a_2-3)(a_1^2+a_2^2-a_1-a_2)-2(a_1+a_2-2)a_1a_2=0$$
Using $S:=a_1+a_2$ and $P:=a_1a_2$, we have $\mathcal F=S^3-4S^2+3S-(4S-10)P$.
If $S, P$ are integers that satisfy $\mathcal F=0$, then $2S-5\, | \, S^3-4S^2+3S$. Modulo $2S-5$, we have $2S\equiv 5$, so 
$2S-5\, | \, 8S^3-32S^2+24S=(2S)^3-8(2S)^2+12(2S)\equiv 
5^3-8\cdot 5^2+12\cdot 5=15$ implies that $2S-5\, | \, 15$. Keeping in mind that $S=a_1+a_2\ge 2$, the possible values for $2S-5$ are $2S-5\in \{-1, 1, 3, 5, 15\}$, so $S\in \{2, 3, 4, 5, 10\}$. For each of these possible values of $S$, the corresponding values of $P=(S^3-4S^2+3S)/(4S-10)$ are $1, 0, 2, 4, 21$ and the resulting $(a_1, a_2)$ are $(1,1), (0,3), (\text{non-integer}), (1,4), (3,7)$. The only solution satisfying $a_1\ge 2$ is $(a_1, a_2)=(3, 7)$, which gives $(a_1, a_2, a_3)=(3, 7, 14)$.

Now assume $a_1=1$.
The coefficient of $x^{t-2}y^{t-1}$ in $\tau^2D(x+y)^{a_3}$ is 
$$
\mathcal C:=\mathrm{cff}(x, \tau^2D)\binom{a_3}{\frac{a_3-2}{2}}+\mathrm{cff}(y, \tau^2D)\binom{a_3}{\frac{a_3}{2}}.
$$
Let $A:=\binom{a_3}{\frac{a_3}{2}}$. Then $\displaystyle \binom{a_3}{\frac{a_3-2}{2}}=A\frac{a_3}{a_3+2}$ and we have 
$$
\mathcal C=0\Leftrightarrow (a_2^2-a_2)a_3-(a_1a_2+a_2)(a_3+2)=0
$$
We divide by $a_2$ and substitute $a_1=1$ and $a_3=2a_2-4$ to see that 
$$
\mathcal C=0\Leftrightarrow 2(a_2-4)(a_2-1)=0
$$
One solution is $a_2=1$, which is not possible because it implies $a_3=-2$. The other solution is $a_2=4$, which gives $(a_1, a_2, a_3)=(1, 4, 4).$

{\bf The case $a=3$:} We need to consider the cases $a_1\ge 3$, $a_1=2$ and $a_1=1$ separately.

Assume $a_1\ge 3$. The coefficient of $x^{t-1}y^{t-1}$ in $(\tau^3D)(x+y)^{a_3}$ is
$$
\mathcal C:=\mathrm{cff}(x^3, \tau^3D)\binom{a_3}{\frac{a_3-3}{2}} +\mathrm{cff}(x^2y, \tau^3D)\binom{a_3}{\frac{a_3-1}{2}} 
$$
$$
+\mathrm{cff}(xy^2, \tau^3D)\binom{a_3}{\frac{a_3+1}{2}}+\mathrm{cff}(y^3, \tau^3D)\binom{a_3}{\frac{a_3+3}{2}}
$$
Let $A:=\binom{a_3}{\frac{a_3-1}{2}}$ and $B:=\binom{a_3}{\frac{a_3-3}{2}}$.
Note $B=\frac{a_3-1}{a_3+3}A$.

Then 
$$
\mathcal C=(\mathrm{cff}(x^3, \tau^3D)+\mathrm{cff}(y^3, \tau^3D))B +(\mathrm{cff}(x^2y, \tau^3D)+\mathrm{cff}(xy^2, \tau^3D))A=0\Leftrightarrow
$$
$$(\mathrm{cff}(x^3, \tau^3D)+\mathrm{cff}(y^3, \tau^3D))(a_3-1) +(\mathrm{cff}(x^2y, \tau^3D)+\mathrm{cff}(xy^2, \tau^3D))(a_3+3)=0
$$
Substituting $a_3=2(a_1+a_2)-9$ and dividing by 2, we have 
$$
\mathcal C=0\Leftrightarrow \mathcal G(a_1, a_2):=(a_1+a_2-5)(a_1^3-a_2^3-3a_1^2+3a_2^2+2a_1-2a_2)+(a_1+a_2-3)(-3a_1^2a_2+3a_1a_2^2)=0.
$$
Note that $\mathcal G(a_1, a_2)$ is divisible by $a_1-a_2$, so every $a_1=a_2$ is a solution for $\mathcal G(a_1, a_2)=0$.
Dividing by $a_1-a_2$ leaves $$\mathcal{F}(a_1, a_2):=a_1^3-a_1^2a_2-a_1a_2^2+a_2^3-8a_1^2-2a_1a_2-8a_2^2+17a_1+17a_2-10.$$
Using $S:=a_1+a_2$ and $P:=a_1a_2$, we have $$\mathcal F= S(S^2-3P)-SP-8(S^2-2P)-2P+17S-10 $$
so $\mathcal F=0\Leftrightarrow S^3-8S^2+17S-10=4SP-14P$.
This implies $2S-7\, | \, S^3-8S^2+17S-10$. Modulo $2S-7$, we have $2S\equiv 7$, so 
$$
2S-7\, | \, 8S^3-64S^2+17\cdot 8 S -80 = (2S)^3-16\cdot (2S)^2+68(2S)-80 \equiv 7^3-16\cdot 7^2+68\cdot 7-80 = -45.$$
It remains to check the divisors of 45 as possible values of $2S-7$.
As $a_3=2S-9\ge 1$, we have $2S-7\in\{3, 5, 9, 15, 45\}$, i.e. $S\in\{5, 6, 8, 11, 26\}$. The corresponding values of $P=(S^{3}-8S^{2}+17S-10)/(4S-14)$ are given by $0$, $2$, $7$, $18$, $140$ and the resulting $(a_{1},a_{2})$ are: $(0,5), (\text{non-integer}), (1,7), (2,9), (\text{non-integer})$. In particular, there are no integer solutions with $a_1\ge 3$.

Now assume $a_1=2$. We must consider the coefficient of $x^{t-2}y^{t-1}$ in $(\tau^{3}D)(x+y)^{a_{3}}$. This is
\begin{equation*}
    \mathcal{C}:=\text{cff}(x^{2},\tau^{3}D)\binom{a_{3}}{\frac{a_{3}-3}{2}}+\text{cff}(xy,\tau^{3}D)\binom{a_{3}}{\frac{a_{3}-1}{2}}+\text{cff}(y^{2},\tau^{3}D)\binom{a_{3}}{\frac{a_{3}+1}{2}}.
\end{equation*}
Set $A=\binom{a_{3}}{\frac{a_{3}-1}{2}}=\binom{a_{3}}{\frac{a_{3}+1}{2}}$, which makes the binomial coefficient in the first term expressible as $\frac{a_{3}-1}{a_{3}+3}A$. With this identification, we have $\mathcal{C}=0$ if and only if
\begin{equation*}
    (a_{3}-1)\text{cff}(x^{2},\tau^{3}D)+(a_{3}+3)\text{cff}(xy,\tau^{3}D)+(a_{3}+3)\text{cff}(y^{2},\tau^{3}D)=0.
\end{equation*}
As $a_{1}=2$, we have $a_{3}=2a_{2}-5$. Substituting these values into the relation described above, we obtain
\begin{equation*}
    -2a_{2}^{4}+24a_{2}^{3}-58a_{2}^{2}+36a_{2}=0\quad\implies\quad -2a_{2}(a_{2}-1)(a_{2}-2)(a_{2}-9)=0.
\end{equation*}
However, $a_{2}\leq 2$ would imply $a_{3}<0$, an obvious impossibility. Thus, the only  ``viable'' root is $a_{2}=9$, which gives rise to $(a_{1},a_{2},a_{3})=(2,9,13)$.

Finally, assume $a_1=1$. We examine the coefficient of $x^{t-3}y^{t-1}$ in $(\tau^{3}D)(x+y)^{a_{3}}$. This is
\begin{equation*}
    \mathcal{C}:=\text{cff}(x,\tau^{3}D)\binom{a_{3}}{\frac{a_{3}-3}{2}}+\text{cff}(y,\tau^{3}D)\binom{a_{3}}{\frac{a_{3}-1}{2}}.
\end{equation*}
Keeping $A=\binom{a_{3}}{\frac{a_{3}-1}{2}}$, it follows that $\mathcal{C}=0$ precisely when
\begin{equation*}
    (a_{3}-1)\text{cff}(x,\tau^{3}D)+(a_{3}+3)\text{cff}(y,\tau^{3}D)=0.
\end{equation*}
As $a_{1}=1$, we have $a_{3}=2a_{2}-7$. Substituting these values into the relation described above, we obtain
\begin{equation*}
    -2a_{2}^{4}+20a_{2}^{3}-46a_{2}^{2}+28a_{2}=0\quad\implies\quad -2a_{2}(a_{2}-1)(a_{2}-2)(a_{2}-7)=0.
\end{equation*}
However, $a_{2}\leq 3$ would imply $a_{3}<0$, which is absurd. This leaves only the root $a_{2}=7$, corresponding to an $(a_{1},a_{2},a_{3})$ triple with two equal entries, namely $(1,7,7)$.
\end{proof}

We point out that the idea of the proof of \autoref{borderline} can be used to tackle larger values of $a$. For a fixed value of $a$, one has a polynomial $\mathcal F(a_1, a_2)$ constructed as above, which is symmetric in $a_1$ and $a_2$. This can then be expressed a polynomial in the elementary symmetric functions $S=a_1+a_2$ and $P=a_1a_2$. We have implemented the construction for the polynomial obtained in the case $a_1\ge a$ in \texttt{Macaulay2}, using the package ``SymmetricPolynomials,'' and have observed experimentally that when $\mathcal F(a_1, a_2)$ is expressed as a polynomial in $P$ with coefficients in $\F[S]$, the constant term is
\begin{equation*}
    \mathcal{P}(S):=\begin{cases}
        \prod\limits_{i=1}^{a-1}(S-i)\prod\limits_{i=\frac{3a+1}{2}}^{2a-1}(S-i) & \text{for}\ a\ \text{odd}, \\[20pt]
        \prod\limits_{i=0}^{a-1}(S-i)\prod\limits_{i=\frac{3a}{2}}^{2a-1}(S-i) & \text{for}\ a\ \text{even},
    \end{cases}
\end{equation*}
and all the other coefficients of powers of $P$ are multiples of $2S-c(a)$, where $c(a)=3a-2$ for $a$ odd and $c(a)=3a-1$ for $a$ even. If $\mathcal F(a_1, a_2)=0$, it follows that $2S-c(a)$ must divide $\mathcal P(S)$. Similar to the argument in the proof of \autoref{borderline}, this gives a finite list of possible values of $S$. For each possible value of $S$ from this list, $P$ must divide $\mathcal P(S)$. Therefore, we will have a finite list of values of $S$ and $P$ to consider, giving rise to a finite list of values of $a_1, a_2$. Unfortunately, for values of $a\ge 7$ the list of values that need to be checked is too large to be practical. Using \texttt{Macaulay2}, we were able to verify for values of $a$ up to 6 that there are no positive integer solutions for $\mathcal{F}(a_1, a_2)=0$, except for the known exceptions when failure of WLP has been established. For larger values of $a$, we were not able to finish checking all possible values even with the help of \texttt{Macaulay2} (although we have verified all possible values of $S$ up to 1000 for $a$ up to 40). \vspace{15pt}

\begin{center}
    \textbf{\textit{Acknowledgments}}
\end{center}\vspace{5pt}

We wish to credit two resources that we consulted numerous times in the development of these results: the On-Line Encyclopedia of Integer Sequences (\cite{OEIS}) and the computer algebra system \texttt{Macaulay2} (\cite{Grayson-Stillman}). The former resource was particularly helpful in detecting patterns (often involving binomial coefficients) from the many examples we constructed. The latter resource was immensely helpful for computing examples, gathering experimental evidence, and allowing us to much more efficiently formulate the content in this paper. Finally, we also extend our thanks to the two anonymous referees for their feedback, particularly pointing out places where we could improve our exposition.

\printbibliography

\end{document}